\documentclass[11pt,a4paper]{article}

\usepackage[utf8]{inputenc}
\usepackage[T1]{fontenc}
\usepackage{lmodern}
\usepackage{microtype}
\usepackage{amsmath,amssymb,amsthm,mathtools}
\usepackage{mathrsfs}
\usepackage{enumitem}
\usepackage{geometry}
\usepackage{hyperref}
\usepackage{cleveref}
\usepackage{xcolor}
\usepackage{tikz}
\usetikzlibrary{arrows.meta,calc,decorations.markings}

\newtheorem{theorem}{Theorem}[section]
\newtheorem{proposition}[theorem]{Proposition}
\newtheorem{lemma}[theorem]{Lemma}
\newtheorem{corollary}[theorem]{Corollary}

\theoremstyle{definition}
\newtheorem{definition}[theorem]{Definition}
\newtheorem{example}[theorem]{Example}

\theoremstyle{remark}
\newtheorem{remark}[theorem]{Remark}

\newtheorem*{theoremA}{Theorem A}

\newcommand{\C}{\mathbb{C}}
\newcommand{\R}{\mathbb{R}}
\newcommand{\Z}{\mathbb{Z}}

\newcommand{\OO}{\mathcal{O}}
\newcommand{\Crit}{\operatorname{Crit}}

\newcommand{\Inv}{\operatorname{Inv}}
\newcommand{\re}{\operatorname{Re}}
\newcommand{\im}{\operatorname{Im}}
\newcommand{\ind}{\operatorname{ind}}
\newcommand{\rank}{\operatorname{rank}}
\newcommand{\wtH}{\widetilde H}
\newcommand{\CH}{CH}
\newcommand{\df}{\partial f}

\title{\bfseries
The Milnor Fiber as a Local Conley Index
}

\author{
Jean-Paul Brasselet
\and
Nivaldo Grulha
}

\date{}

\begin{document}

\maketitle

\begin{abstract}
Let $f:(\C^n,0)\to(\C,0)$ be a holomorphic germ with an isolated
critical point at the origin, and let $\varphi_f$ be the negative
gradient flow of $\re f$ with respect to a Riemannian metric.  We show
that the origin is an isolated invariant set of $\varphi_f$ and that its
Conley index is the suspension of the Milnor fiber,
\[
h(\varphi_f,\{0\})\;\simeq\;\Sigma F_f\;\simeq\;\bigvee_{\mu(f)}S^n ,
\]
so that $\mu(f)=\rank \CH_n(\varphi_f,\{0\};\Z)$: the Milnor number is
the rank of the homological Conley index of the flow in degree $n$.  The proof
uses continuation to a holomorphic morsification of $f$.  Since $\im f$
is a first integral of the flow, after a rotation of the target through
an angle avoiding finitely many values there are no connecting orbits
between the $\mu(f)$ nondegenerate rest points that result, and the
sum property of the Conley index applies.  McCord's theorem then
recovers the Poincar\'e--Hopf index of $-\nabla\re f$.
\end{abstract}

\section{Introduction and main results}

Let
\[
f:(\C^n,0)\longrightarrow(\C,0)
\]
be a germ of a holomorphic function with an isolated critical point at
the origin.  Milnor's fibration theorem \cite{Milnor1968} associates
with $f$ the Milnor fiber
\[
F_f=f^{-1}(\delta)\cap B_\varepsilon,
\qquad 0<|\delta|\ll\varepsilon\ll 1,
\]
which, in the isolated case, has the homotopy type of a bouquet of
middle-dimensional spheres,
\begin{equation}\label{eq:bouquet}
F_f\simeq\bigvee_{\mu(f)}S^{\,n-1},
\end{equation}
where $\mu(f)$ is the Milnor number.

We write $z=x+iy$ with $x,y\in\R^n$, identify $\C^n\cong\R^{2n}$, and write $f=\re f+i\,\im f$, where $\re f=\operatorname{Re}f$ and $\im f=\operatorname{Im}f$ are real-valued functions of $(x,y)$.  On the dynamical side, the real part
\[
g=\re f:\R^{2n}\longrightarrow\R
\]
has an isolated critical point at the origin, and its negative
gradient flow $\varphi_f$ (with respect to the standard Hermitian
metric) is a local gradient dynamical system.  Conley index
theory \cite{Conley1978,Salamon1985} attaches to every isolated
invariant set of a flow a pointed homotopy type, invariant under
continuation; we determine this invariant for the equilibrium of
$\varphi_f$.

\begin{theoremA}
Let $f:(\C^n,0)\to(\C,0)$, $n\ge 1$, be a holomorphic germ with an
isolated critical point at the origin, and let $\varphi_f$ be the
local negative gradient flow of $\re f$ with respect to a Riemannian
metric on a neighborhood of $0$.  Then $\{0\}$ is an
isolated invariant set of $\varphi_f$, its Conley index does not
depend on the choice of metric, and
\[
h(\varphi_f,\{0\})
\;\simeq\;
\Sigma F_f
\;\simeq\;
\bigvee_{\mu(f)}S^{\,n}.
\]
Consequently,
\[
\CH_k(\varphi_f,\{0\};\Z)\;\cong\;
\begin{cases}
\Z^{\mu(f)}, & k=n,\\[2pt]
0, & k\neq n,
\end{cases}
\qquad\text{and}\qquad
\boxed{\;\mu(f)=\rank \CH_n(\varphi_f,\{0\};\Z).\;}
\]
\end{theoremA}

The suspension arises from the
Cauchy--Riemann equations: every Morse point of the real part of a
holomorphic function has Morse index exactly $n$, so the vanishing
cycles, which live in degree $n-1$, are seen by the dynamics with a
shift of one degree.

\begin{figure}[htb]
\centering
\begin{tikzpicture}[scale=0.8,transform shape,>=Stealth,
 ex/.style={red!70!black,line width=2pt}]
\node at (0,1.9) {\small Milnor fiber $F_f$};
\node at (4.4,1.9) {\small suspension $\Sigma F_f$};
\node at (8.8,1.9) {\small Conley index: $D^n_u/\partial D^n_u$};
\node[rotate=90] at (-2.1,0) {\small $n=1$, $f=z^2$};
\node[rotate=90] at (-2.1,-4.2) {\small $n=2$, $f=z_1^2+z_2^2$};
\draw[thin,gray] (-1.2,0)--(1.2,0);
\fill[blue!60!black] (-0.8,0) circle (2.4pt); \fill[blue!60!black] (0.8,0) circle (2.4pt);
\node[below] at (-0.8,-0.05) {\scriptsize $-1$}; \node[below] at (0.8,-0.05) {\scriptsize $1$};
\node at (0,-1.3) {\small $\simeq S^0$};
\draw[->,thick] (1.6,0)--(2.8,0) node[midway,above] {\scriptsize $\Sigma$};
\draw (4.4,0) circle (0.8);
\fill (4.4,0.8) circle (1.8pt); \fill (4.4,-0.8) circle (1.8pt);
\fill[blue!60!black] (3.6,0) circle (2.4pt); \fill[blue!60!black] (5.2,0) circle (2.4pt);
\node at (4.4,-1.3) {\small $\simeq S^1$};
\node at (6.6,0) {$\simeq$};
\draw[thick] (8.8,-0.8)--(8.8,0.8);
\fill[red!70!black] (8.8,0.8) circle (3pt); \fill[red!70!black] (8.8,-0.8) circle (3pt);
\node[right] at (8.85,0) {\scriptsize $D^1_u$};
\node at (8.8,-1.3) {\small $D^1_u/\partial D^1_u$};
\begin{scope}[yshift=-4.2cm]
 \draw (-0.8,0.8)--(-0.8,-0.8); \draw (0.8,0.8)--(0.8,-0.8);
 \draw (0,0.8) ellipse (0.8 and 0.25);
 \draw (-0.8,-0.8) arc (180:360:0.8 and 0.25);
 \draw[dashed] (-0.8,-0.8) arc (180:0:0.8 and 0.25);
 \draw[red!70!black,thick] (-0.8,0) arc (180:360:0.8 and 0.25);
 \draw[red!70!black,thick,dashed] (-0.8,0) arc (180:0:0.8 and 0.25);
 \node at (0,-1.3) {\small $\simeq S^1$ (vanishing cycle)};
 \draw[->,thick] (1.6,0)--(2.8,0) node[midway,above] {\scriptsize $\Sigma$};
 \draw (4.4,0) circle (0.9);
 \draw (3.5,0) arc (180:360:0.9 and 0.25);
 \draw[dashed] (3.5,0) arc (180:0:0.9 and 0.25);
 \node at (4.4,-1.3) {\small $\simeq S^2$};
 \node at (6.6,0) {$\simeq$};
 \fill[blue!6] (8.8,0) circle (0.8);
 \draw[red!70!black,line width=2pt] (8.8,0) circle (0.8);
 \node at (8.8,0) {\scriptsize $D^2_u$};
 \node at (8.8,-1.3) {\small $D^2_u/\partial D^2_u$};
\end{scope}
\end{tikzpicture}
\caption{Theorem~A for $f=z_1^2+\dots+z_n^2$, $n=1,2$: $F_f$, $\Sigma F_f$, and $N_1/N_2\simeq D^n_u/\partial D^n_u$ for the block $D^n_s\times D^n_u$ (exit set $D^n_s\times\partial D^n_u$).}
\label{fig:fiber-index}
\end{figure}
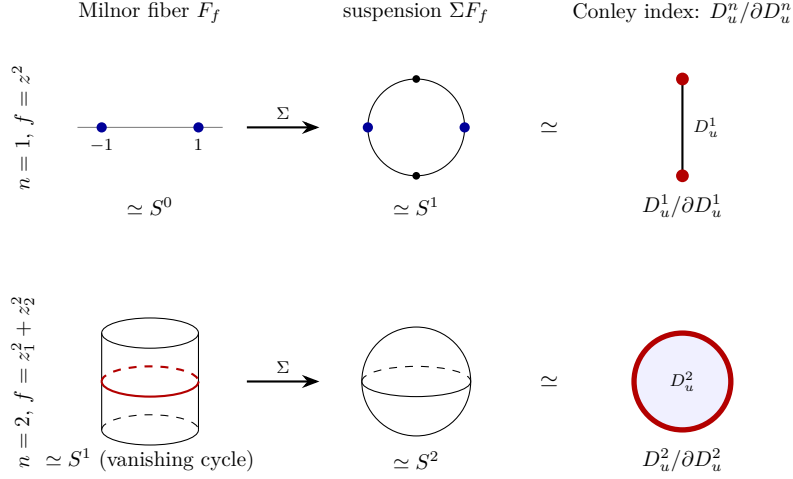

For the standard Hermitian metric one has, along the flow lines,
\begin{equation}\label{eq:fdot}
\frac{d}{dt}\,f(z(t))=-\bigl\|\df(z(t))\bigr\|^{2}\in\R_{\le 0},
\end{equation}
so $\im f$ is a first integral and $\re f$ is a strict Lyapunov
function, which isolates the equilibrium.  A holomorphic morsification
of $f$ deforms the flow into one with $\mu(f)$ nondegenerate rest
points, all of index $n$.  A connecting orbit between two of them,
$p_i,p_j$, lies on a horizontal line in the $f$-plane, which forces
$f(p_i)-f(p_j)$ to be a positive real number.  A rotation
$f\mapsto e^{-i\theta}f$ of the target through an angle avoiding
finitely many values excludes this for every pair, so the invariant set
of the morsified flow is a disjoint union of $\mu(f)$ hyperbolic rest
points, and the sum property of the Conley index gives the wedge.  The
case $n=1$ is included: $F_f$ is a finite set of points and $\Sigma F_f$
is a wedge of circles.

\subsection*{Relation with previous work}

For gradient flows, the homological Conley index of an isolated critical point
is given by the critical groups of Gromoll--Meyer
\cite{GromollMeyer1969}: in Chang's formulation
\cite[Def.~4.1, Def.~5.1 and Thm.~5.2]{Chang1993} the critical groups
are the homology of a Gromoll--Meyer pair, a neighbourhood with the
mean value property and its exit set, which are the requirements
defining an index pair for the negative gradient flow.  The homotopy type of the quotient of such a pair by its exit set is what Goresky--MacPherson call the Conley Morse index: for a function with a critical point they define Morse data $(A,B)$ through sublevel sets, with no reference to a flow, and observe that the homotopy type of $A/B$ is well defined, after Conley \cite{Conley1972} \cite[Part~I, \S3.3]{GoreskyMacPherson1988}.  On a manifold with a single stratum, as here, their local Morse data \cite[\S3.5.2, Thm.~3.5.4]{GoreskyMacPherson1988} is Morse data in this sense, and the normal Morse data \cite[\S3.6]{GoreskyMacPherson1988} of the normal slice is trivial.  In the holomorphic setting Theorem~A determines the Conley
index of the gradient flow of $\re f$ as $\Sigma F_f$, for every
$n\ge1$ and every Riemannian metric, and realizes it, by continuation
to a holomorphic morsification, as the wedge of the $\mu(f)$ hyperbolic
rest points of index $n$; the shift of degree between $F_f$ and the
index is explained by \Cref{lem:index-n}.  For homogeneous potentials,
Conley's computation of the index through the exit set
$\{V\ge0\}\cap S$ of the unit ball
\cite[Ch.~III, \S6.2; Ch.~IV, \S1.5A]{Conley1978} gives the
corresponding special case (\Cref{prop:homog}).

Since the Conley index determines the Poincar\'e--Hopf index
\cite{McCord1989,RazvanFotouhi2008}, Theorem~A also yields that of
$-\nabla\re f$ (\Cref{sec:homology}).

\section{Background: Milnor fibrations and the Conley index}
\label{sec:background}

\subsection{Milnor fibrations and morsifications}
\label{sec:milnor}

Throughout, $f:(\C^n,0)\to(\C,0)$ is holomorphic with an isolated
critical point at $0$.  We write $z=(z_1,\dots,z_n)$ with $z_j=x_j+iy_j$, identify $\C^n$ with $\R^{2n}$ through $(x,y)$, and decompose $f(z)=\re f(z)+i\,\im f(z)$ into its real and imaginary parts $\re f=\operatorname{Re}f$ and $\im f=\operatorname{Im}f$, which are real-valued functions on $\R^{2n}$.  We write
\[
\df=\Bigl(\frac{\partial f}{\partial z_1},\dots,
\frac{\partial f}{\partial z_n}\Bigr)
\]
for the holomorphic gradient, and
$\|\df\|^2=\sum_j|\partial f/\partial z_j|^2$.
Suspensions are taken in the pointed category: $\Sigma$ denotes the
reduced suspension of a pointed CW complex.  For (well-)pointed
CW complexes the quotient map from the unreduced to the reduced
suspension collapses the contractible subcomplex $\{x_0\}\times I$ and
is a homotopy equivalence, so reduced and unreduced suspensions agree here.  For $n=1$,
$F_f$ is a set of $\mu(f)+1$ points and either suspension is a wedge
of $\mu(f)$ circles.

\begin{theorem}[Milnor \cite{Milnor1968}]
\label{thm:milnor-fibration}
For sufficiently small $\varepsilon>0$ and $0<\delta\ll\varepsilon$,
the restriction
\[
f:
B_\varepsilon\cap f^{-1}\bigl(\overline D_\delta\setminus\{0\}\bigr)
\longrightarrow
\overline D_\delta\setminus\{0\}
\]
is a smooth locally trivial fibration, the \emph{Milnor fibration in
the tube}, and its fiber $F_f=f^{-1}(\delta)\cap B_\varepsilon$ is
the Milnor fiber.  Moreover every fiber $f^{-1}(w)$,
$|w|\le\delta$, meets $\partial B_\varepsilon$ transversally.
\end{theorem}

\begin{figure}[htb]
\centering
\begin{tikzpicture}[>=Stealth, scale=0.85]
\coordinate (C) at (0,0);
\draw[thick] (C) circle (2.0);
\node at (0,2.35) {$\mathbb S^{2n-1}_\varepsilon$};

\foreach \a in {35,145,250}{
  \filldraw[black] ($(C)+(\a:2.0)$) circle (1.6pt);
}
\node[align=center] at (-1.65,-1.55)
  {\scriptsize $K_f=f^{-1}(0)\cap\mathbb S^{2n-1}_\varepsilon$};

\draw[blue!60!black,thick]
  (0.9,1.55) .. controls (1.5,0.4) and (1.5,-0.4) .. (0.9,-1.55);
\node[blue!60!black] at (1.85,0) {\scriptsize $F_\theta$};

\draw[->,thick] (2.6,0) -- (4.1,0);
\node[above] at (3.35,0.05) {$f/|f|$};

\coordinate (T) at (5.6,0);
\draw[thick] (T) circle (1.1);
\node at (5.6,1.45) {$\mathbb S^1$};
\filldraw[blue!60!black] ($(T)+(0:1.1)$) circle (1.8pt);
\node[right,blue!60!black] at ($(T)+(0:1.1)$) {\scriptsize $e^{i\theta}$};

\draw[dashed,blue!40!black] (0.9,1.55) .. controls (3.2,2.3) .. ($(T)+(0:1.1)$);
\draw[dashed,blue!40!black] (0.9,-1.55) .. controls (3.2,-2.3) .. ($(T)+(0:1.1)$);
\end{tikzpicture}
\caption{The Milnor fibration: $f/|f|$ fibers $\mathbb S^{2n-1}_\varepsilon\setminus K_f$ over $\mathbb S^1$ with fiber $F_f$.}
\label{fig:fibration}
\end{figure}
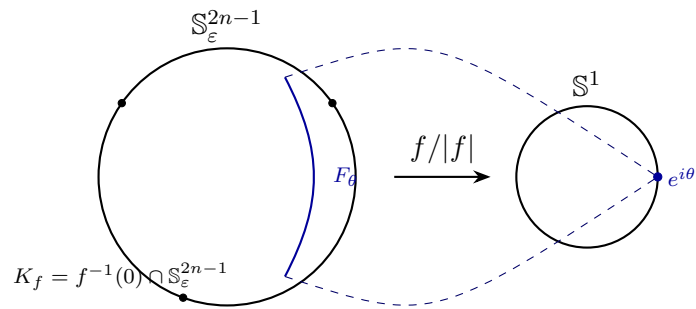

\begin{theorem}[{\cite{Milnor1968}; see also
\cite{ArnoldGuseinZadeVarchenko1988}}]
\label{thm:wedge}
If $0$ is an isolated critical point, then
$F_f\simeq\bigvee_{\mu(f)}S^{\,n-1}$, where
\[
\mu(f)=\dim_\C\OO_{\C^n,0}\Big/
\Bigl(\tfrac{\partial f}{\partial z_1},\dots,
\tfrac{\partial f}{\partial z_n}\Bigr).
\]
For $n=1$ the statement reads: $F_f$ is a set of $\mu(f)+1$ points.
\end{theorem}

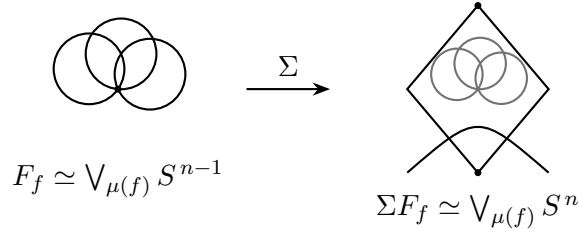
\begin{figure}[htb]
\centering
\begin{tikzpicture}[>=Stealth, scale=0.85]
\coordinate (B) at (0,0);
\foreach \ang in {25,85,145}{
  \draw[thick] ($(B)+(\ang:0.55)$) circle (0.55);
}
\filldraw (B) circle (1.3pt);
\node at (0,-1.4) {$F_f\simeq\bigvee_{\mu(f)} S^{\,n-1}$};

\draw[->,thick] (2.0,0) -- (3.3,0);
\node[above] at (2.65,0.05) {$\Sigma$};

\coordinate (S) at (5.6,0);
\draw[thick] (S) ++(0,1.3) -- ++(-1.1,-1.3) -- ++(1.1,-1.3)
  -- ++(1.1,1.3) -- cycle;
\draw[thick] ($(S)+(-1.1,-1.3)$) .. controls ($(S)+(0,-0.35)$)
  .. ($(S)+(1.1,-1.3)$);
\foreach \ang in {25,85,145}{
  \draw[thick,black!55] ($(S)+(\ang:0.4)$) circle (0.4);
}
\filldraw (S) ++(0,1.3) circle (1.3pt);
\filldraw (S) ++(0,-1.3) circle (1.3pt);
\node at (5.6,-1.9) {$\Sigma F_f\simeq\bigvee_{\mu(f)} S^{\,n}$};
\end{tikzpicture}
\caption{$\Sigma$ turns the wedge of $(n-1)$-spheres $F_f$ into a wedge of $n$-spheres (\Cref{thm:wedge}).}
\label{fig:suspension}
\end{figure}

\begin{definition}
\label{def:morsification}
A \emph{morsification} of $f$ is a holomorphic family
$f_t$, $t\in\C$, $|t|$ small, with $f_0=f$, defined on a fixed
neighborhood of $0$, such that for all sufficiently small $t\neq0$
the function $f_t$ has only nondegenerate (Morse) critical points in
that neighborhood.
\end{definition}

Morsifications exist.  The map $\partial f:(\C^n,0)\to(\C^n,0)$ is finite
of local degree $\mu(f)$
\cite[Ch.~I, \S5.1--5.2]{ArnoldGuseinZadeVarchenko1985}; its critical
locus $\{\det\operatorname{Hess}f=0\}$ is a proper analytic subset, so
the discriminant $\Delta\subset(\C^n,0)$, its image, is a proper
analytic hypersurface germ.  Take $a\in\C^n$ outside the tangent cone of
$\Delta$ at $0$ (a proper algebraic cone, so $a$ is generic; recall that
$a\in C_0\Delta$ means $x_k/\|x_k\|\to a/\|a\|$ for some
$x_k\in\Delta\setminus\{0\}$, $x_k\to0$) and put
$f_t=f-t\sum_ja_jz_j$.  The critical points of $f_t$ are the solutions
of $\partial f(z)=ta$.  If $t_ka\in\Delta$ for a sequence $t_k\to0$ in
$\C$, pass to a subsequence with $t_k/|t_k|\to\lambda\in S^1$; then
$t_ka/\|t_ka\|\to\lambda a/\|a\|$, so $\lambda a\in C_0\Delta$, and
since $C_0\Delta$ is a complex cone, $a\in C_0\Delta$, a
contradiction.  Hence the complex line $\C a$ meets
$\Delta$ only at $0$ near the origin, and $ta\notin\Delta$ for
$0<|t|\ll1$.  Hence all solutions near $0$ are regular points of
$\partial f$, that is, nondegenerate critical points of $f_t$.  The
family $f_t$ is a morsification; see also
\cite{ArnoldGuseinZadeVarchenko1988} and
\cite[Ch.~I, \S8.3]{ArnoldGuseinZadeVarchenko1985} for versal
unfoldings.  The
critical points of a morsification satisfy the following count.

\begin{theorem}[{\cite{ArnoldGuseinZadeVarchenko1988}; see also \cite[Ch.~I, \S5]{ArnoldGuseinZadeVarchenko1985}}]
\label{thm:morsification-count}
Let $f_t$ be a morsification of $f$ and let $\varepsilon>0$ be small
enough that $\Crit(f)\cap \overline B_\varepsilon=\{0\}$.  Then for
all sufficiently small $t\neq0$ the function $f_t$ has exactly
$\mu(f)$ critical points
$p_1(t),\dots,p_{\mu(f)}(t)$ in $B_\varepsilon$, all nondegenerate,
all contained in an arbitrarily small ball around the origin as
$t\to0$, so they lie in the open ball
$B_{\varepsilon/2}$ for $|t|$ small.
\end{theorem}

\begin{lemma}
\label{lem:index-n}
Let $p$ be a nondegenerate critical point of a holomorphic function
$F$ on an open set of $\C^n$.  Then $p$ is a nondegenerate critical
point of $\re F:\R^{2n}\to\R$ of Morse index exactly $n$.
\end{lemma}

\begin{proof}
By the holomorphic Morse lemma we may choose local holomorphic
coordinates in which $F=F(p)+\sum_{j=1}^n z_j^2$.  Writing
$z_j=x_j+iy_j$,
\[
\re F=\re F(p)+\sum_{j=1}^n\bigl(x_j^2-y_j^2\bigr),
\]
which is a nondegenerate real quadratic form of signature $(n,n)$,
hence of Morse index $n$.  Nondegeneracy and the Morse index of
$\re F$ are invariant under diffeomorphism, so the statement holds in the original coordinates.
\end{proof}

\subsection{The Conley index}
\label{sec:conley}

Let $\varphi$ be a (local) flow on a locally compact metric space: a
rule $\varphi_t$ that moves each point $x$ along a curve
$t\mapsto\varphi_t(x)$, its \emph{orbit}.  An orbit is \emph{complete}
if it is defined for all $t\in\R$, so it records the whole past and
future of the point.  For the flow of a vector field $X$, that is, of the equation
$\dot z=X(z)$, an orbit is a solution $z(t)$ of this equation; we
call it an $X$-\emph{orbit} to indicate which field it belongs to.  A \emph{rest point} (equilibrium) is a point with
$\varphi_t(x)=x$ for all $t$, that is, a zero of the vector field.  A
set is \emph{invariant} if it contains the complete orbit of each of
its points, and for a compact set $N$ the \emph{maximal invariant set}
$\Inv(N,\varphi)$ below consists of the points whose complete orbit
stays in $N$.
A compact set $N$ is an \emph{isolating neighborhood} if
\[
\Inv(N,\varphi)
=\{x\in N:\varphi_\R(x)\subset N\}
\subset\operatorname{int}N,
\]
and $S=\Inv(N,\varphi)$ is then an \emph{isolated invariant set}
\cite[Ch.~III, \S3.2]{Conley1978}.  In words: no complete orbit in $N$
touches the boundary of $N$, so $S$ stays away from $\partial N$ and
remains stable under small perturbations of the flow.  Any smaller compact neighborhood of
$S$ in $N$ is again isolating, and disjoint isolated invariant sets
have disjoint isolating neighborhoods whose union isolates their
union \cite[Ch.~III, \S3.2B,C]{Conley1978}.
An \emph{index pair} for $S$ in $N$ is an ordered pair $(N_1,N_2)$ of
compact sets, $N_2\subset N_1\subset N$, both positively invariant
relative to $N$, with $S\subset\operatorname{int}(N_1\setminus N_2)$,
and such that if $x\in N_1$, $t>0$ and $x\cdot t\notin N$, then
$x\cdot[0,t']\subset N_1$ and $x\cdot t'\in N_2$ for some $t'\in[0,t]$
\cite[Ch.~III, \S4.1]{Conley1978}.  Here $N_1$ is a region around $S$ and
$N_2\subset N_1$ is its \emph{exit set}, the part of the boundary
through which orbits leave; ``positively invariant relative to $N$''
means that an orbit starting in the set stays in it as long as it stays
in $N$.  Collapsing $N_2$ to a point records how the flow leaves
$N_1$ near $S$.  Index pairs exist in every
neighborhood of $S$ in $N$ \cite[Ch.~III, \S4.1D]{Conley1978}, and the pointed homotopy type
\[
h(\varphi,S)=\bigl[\,N_1/N_2\bigr]
\]
is independent of the choice of index pair
\cite[Ch.~III, \S5.1--5.2]{Conley1978}; it is the \emph{Conley index}
(Conley's homotopy index) of $S$.  Its reduced singular homology is the
\emph{homological Conley index} $\CH_*(\varphi,S)=\wtH_*(N_1/N_2)$.

\begin{figure}[htb]
\centering
\begin{tikzpicture}[scale=0.9,transform shape,>=Stealth,
 fl/.style={blue!60!black,thick,postaction={decorate},decoration={markings,mark=at position 0.55 with {\arrow{Stealth[length=2mm]}}}}]
\def\saddle{%
 \foreach \sx/\sy in {1/1,-1/1,-1/-1,1/-1}{
  \foreach \c in {0.25,0.6}{
   \draw[fl] plot[domain=1.9:{\c/1.9}] ({\sx*\x},{\sy*\c/\x});
  }}
 \draw[fl] (1.9,0)--(0,0); \draw[fl] (-1.9,0)--(0,0);
 \draw[fl] (0,0)--(0,1.9); \draw[fl] (0,0)--(0,-1.9);}
\begin{scope}
 \fill[blue!6] (0,0) circle (1.3);
 \begin{scope}\clip (0,0) circle (1.3);\saddle\end{scope}
 \draw[very thick] (0,0) circle (1.3);
 \fill (0,0) circle (2.4pt);
 \node[below left] at (0,0) {$S$};
 \node at (-1.15,1.2) {$N$};
 \node[align=center] at (0,-2.0) {\small isolating: $S\subset\operatorname{int}N$};
\end{scope}
\begin{scope}[xshift=5.2cm]
 \fill[blue!6] (1.1,0) circle (1.1);
 \begin{scope}\clip (1.1,0) circle (1.1);\saddle\end{scope}
 \draw[very thick] (1.1,0) circle (1.1);
 \fill (0,0) circle (2.4pt);
 \node[left] at (0,0) {$S$};
 \node at (2.0,1.0) {$N$};
 \node[align=center] at (1.1,-2.0) {\small not isolating: $S\cap\partial N\neq\emptyset$};
\end{scope}
\end{tikzpicture}
\caption{Isolating (left) and non-isolating (right) neighborhoods of $S=\{0\}$ for a saddle.}
\label{fig:isolating-nbhd}
\end{figure}
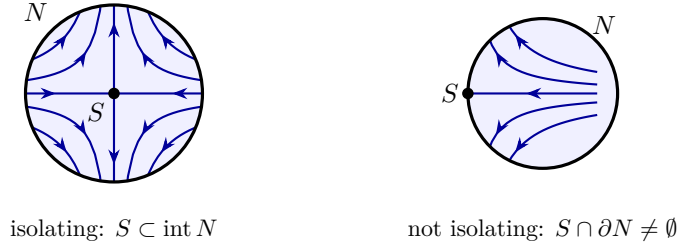

We use the following structural properties.

The next property says that the index can be followed along a family of
flows, as long as the same set $N$ remains isolating (\emph{continuation}).

A \emph{continuous family of flows} $\{\varphi^\lambda\}_{\lambda\in\Lambda}$,
with $\Lambda$ a topological space, means the flows of equations
$\dot x=F(x,\lambda)$ with $F$ continuous in $(x,\lambda)$ and Lipschitz
in $x$ \cite[Ch.~IV, \S1.1]{Conley1978}.

\begin{theorem}[Continuation]
\label{thm:continuation}
Let $\{\varphi^\lambda\}_{\lambda\in[0,1]}$ be a continuous family of
flows, and let $N$ be a compact set which is an isolating neighborhood
for $\varphi^\lambda$ for every $\lambda\in[0,1]$.  Then
\[
h\bigl(\varphi^0,\Inv(N,\varphi^0)\bigr)
\;\simeq\;
h\bigl(\varphi^1,\Inv(N,\varphi^1)\bigr).
\]
\end{theorem}

This is the standard continuation theorem for the Conley index
\cite[Ch.~IV, \S1.4]{Conley1978}, \cite{Salamon1985}.

\begin{proposition}[Openness of isolation]
\label{prop:openness}
Let $\{\varphi^\lambda\}_{\lambda\in\Lambda}$ be a continuous family of
flows and let $N$ be a compact set which is an isolating neighborhood for
$\varphi^{\lambda_0}$.  Then $N$ is an isolating neighborhood for
$\varphi^{\lambda}$ for all $\lambda$ in a neighborhood of $\lambda_0$.
\end{proposition}

See \cite[Ch.~IV, \S1.2A, \S1.3]{Conley1978}.

\begin{theorem}[Sum property \cite{Conley1978}]
\label{thm:sum}
If $S=S_1\sqcup\dots\sqcup S_k$ is a disjoint union of isolated
invariant sets, then $S$ is isolated and
\[
h(\varphi,S)\;\simeq\;\bigvee_{j=1}^{k}h(\varphi,S_j)
\]
\cite[Ch.~I, \S4.1; Ch.~III, \S6.1D]{Conley1978}.
\end{theorem}

\begin{theorem}[Hyperbolic rest points \cite{Conley1978}]
\label{thm:hyperbolic}
If $p$ is a hyperbolic equilibrium of a smooth flow with unstable
dimension $\lambda$ (the linearization has no eigenvalue on the
imaginary axis, and exactly $\lambda$ eigenvalues with positive real
part, the directions along which orbits move away from $p$), then $h(\varphi,\{p\})\simeq S^\lambda$
\cite[Ch.~I, \S4.3; details in Ch.~IV, \S1.5A]{Conley1978}.  In
particular, for the negative gradient flow of a Morse function, the
index of a critical point of Morse index $\lambda$ is $S^\lambda$.
\end{theorem}

\begin{figure}[htb]
\centering
\begin{tikzpicture}[scale=0.9,transform shape,>=Stealth,
 fl/.style={blue!60!black,thick,postaction={decorate},decoration={markings,mark=at position 0.55 with {\arrow{Stealth[length=2mm]}}}},
 ex/.style={red!70!black,line width=2pt}]
\begin{scope}
 \fill[blue!6] (0,0) circle (1.2);
 \foreach \a in {0,60,...,300}{\draw[fl] (\a:1.2)--(0,0);}
 \draw[thick] (0,0) circle (1.2);
 \fill (0,0) circle (2pt);
 \node at (0,-1.9) {\small $\lambda=0$: $N_2=\emptyset$};
 \node at (0,-2.4) {\small $N_1/N_2\simeq S^0$};
\end{scope}
\begin{scope}[xshift=4.2cm]
 \fill[blue!6] (-1.2,-1.2) rectangle (1.2,1.2);
 \foreach \sx/\sy in {1/1,-1/1,-1/-1,1/-1}{
  \foreach \c in {0.3,0.7}{
   \draw[fl] plot[domain=1.2:{\c/1.2}] ({\sx*\x},{\sy*\c/\x});}}
 \draw[fl] (1.2,0)--(0,0); \draw[fl] (-1.2,0)--(0,0);
 \draw[fl] (0,0)--(0,1.2); \draw[fl] (0,0)--(0,-1.2);
 \draw[thick] (-1.2,-1.2) rectangle (1.2,1.2);
 \draw[ex] (-1.2,1.2)--(1.2,1.2); \draw[ex] (-1.2,-1.2)--(1.2,-1.2);
 \fill (0,0) circle (2pt);
 \node at (0,-1.9) {\small $\lambda=1$: $N_2=$ exit set};
 \node at (0,-2.4) {\small $N_1/N_2\simeq S^1$};
\end{scope}
\begin{scope}[xshift=8.4cm]
 \fill[blue!6] (0,0) circle (1.2);
 \foreach \a in {0,60,...,300}{\draw[fl] (0,0)--(\a:1.2);}
 \draw[ex] (0,0) circle (1.2);
 \fill (0,0) circle (2pt);
 \node at (0,-1.9) {\small $\lambda=2$: $N_2=\partial N_1$};
 \node at (0,-2.4) {\small $N_1/N_2\simeq S^2$};
\end{scope}
\end{tikzpicture}
\caption{Index pairs $(N_1,N_2)$ of hyperbolic rest points; $N_2$ (red) is the exit set.}
\label{fig:index-pairs}
\end{figure}
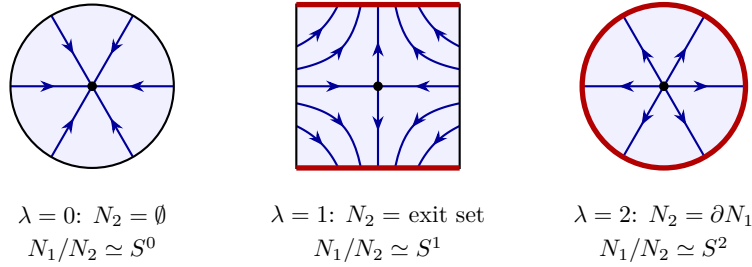

\section{The real gradient flow: proof of Theorem A}
\label{sec:proofA-group}

\subsection{Gradient dynamics of the real part}
\label{sec:realpart}

Let $g=\re f$.  Its \emph{negative gradient flow} is the flow of the
equation $\dot z=-\nabla g(z)$: each point moves in the direction of
steepest descent of $g$, so $g$ decreases along orbits and the rest
points are the critical points of $g$.  The gradient depends on a choice
of Riemannian metric; consider first the standard Hermitian metric on
$\C^n$.  Identifying $T\C^n\cong\C^n$, the Cauchy--Riemann equations
give
\[
\nabla g(z)=\overline{\df(z)},
\]
so the negative gradient flow is the solution flow of
\begin{equation}\label{eq:flow}
\dot z=-\,\overline{\df(z)}.
\end{equation}

\begin{lemma}[Target dynamics]
\label{lem:target}
Along any solution $z(t)$ of \eqref{eq:flow},
\[
\frac{d}{dt}\,f(z(t))
=\df(z(t))\cdot\dot z(t)
=-\bigl\|\df(z(t))\bigr\|^{2}
\;\in\;\R_{\le0}.
\]
Consequently $\im f$ is a first integral of the flow, $\re f$ is
strictly decreasing along nonconstant orbits, and the image under
$f$ of every orbit travels along a horizontal line in $\C$, moving
from right to left with speed $\|\df\|^2$.
\end{lemma}

\begin{proof}
Substitute \eqref{eq:flow} and use $\df\cdot\overline{\df}=\|\df\|^2$.
\end{proof}

A \emph{first integral} is a function that is constant along every
orbit, and a \emph{Lyapunov function} is one that does not increase
along orbits; it is \emph{strict} when it decreases strictly along every
nonconstant orbit.  Thus $\im f$ confines the image of each orbit to a
horizontal line, and $\re f$ forces it to move in one direction only.

\begin{figure}[htb]
\centering
\begin{tikzpicture}[>=Stealth, scale=0.85]
\draw[->,thick] (-3.2,0) -- (3.2,0) node[right] {$\operatorname{Re} w$};
\draw[->,thick] (0,-2.0) -- (0,2.0) node[above] {$\operatorname{Im} w$};

\foreach \y in {-1.4,-0.85,-0.3,0.25,0.8,1.35}{
  \draw[->,blue!55!black,thick] (2.9,\y) -- (-2.6,\y);
}

\filldraw[black] (2.2,1.35) circle (1.6pt);
\node[above] at (2.2,1.55) {\scriptsize $f(z(0))$};
\filldraw[black] (-1.6,1.35) circle (1.6pt);
\node[above] at (-1.6,1.55) {\scriptsize $f(z(t))$};
\draw[->,black!70,thin] (2.15,1.3) .. controls (0.3,1.0) .. (-1.55,1.3);

\node at (0,-2.5)
  {$\dfrac{d}{dt}f(z(t))=-\|\partial f(z(t))\|^2\in\mathbb R_{\le 0}$};
\end{tikzpicture}
\caption{The image $f(z(t))$ moves leftward at constant $\operatorname{Im}w$ (\Cref{lem:target}).}
\label{fig:dynamics}
\end{figure}
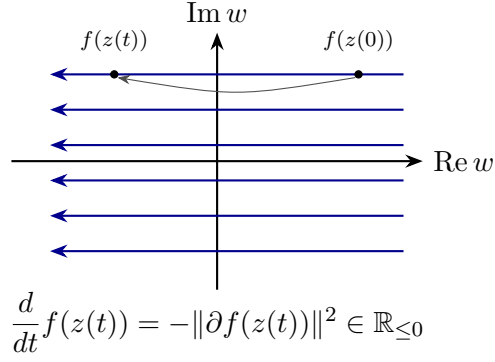

\begin{remark}
\label{rem:hermitian}
Let $h$ be any Hermitian metric (that is, $J$-invariant) and
$X=-\nabla_h g$.  The Cauchy--Riemann equations give
$d(\im f)=-dg\circ J$, hence
$d(\im f)(\nabla_h g)=-dg(J\nabla_h g)=-h(\nabla_h g,J\nabla_h g)=0$,
since $J$ is skew-adjoint for the Hermitian metric $h$,
and $dg(\nabla_h g)=\|dg\|_h^2$.  Therefore
\[
\frac{d}{dt}f(z(t))=-\|dg(z(t))\|_h^2\in\R_{\le0}
\]
along $X$-orbits for every Hermitian metric: $\im f$ is a first
integral and $\re f$ a strict Lyapunov function, exactly as in
\Cref{lem:target}.  For a general Riemannian metric only the Lyapunov
property is available: $\im f$ need not be a first integral.
\end{remark}

\begin{lemma}[Isolation, uniform in the metric]
\label{lem:isolation}
Let $g=\re f$, let $B_\varepsilon$ denote the open ball of radius $\varepsilon$ at
$0$ in $\C^n\cong\R^{2n}$ and $\overline B_\varepsilon$ its closure, with
$\varepsilon>0$ small enough that
$\Crit(f)\cap\overline B_\varepsilon=\{0\}$ (possible since the
critical point is isolated).  Let $\{h_s\}_{s\in[0,1]}$ be a family of
smooth Riemannian metrics on a neighborhood of
$\overline B_\varepsilon$, depending continuously on $s$ in the
$C^1$ topology, let $X_s=-\nabla^{h_s}g$ be the negative
gradient field of $g$ for $h_s$, and let $\varphi^s$ be its flow.
Then, for every $s\in[0,1]$:
\begin{enumerate}[label=\emph{(\roman*)},leftmargin=*]
\item the critical set of $g$ does not depend on the metric:
$\Crit(g)\cap\overline B_\varepsilon=\{0\}$, and $X_s(z)=0$ if and only
if $z\in\Crit(g)$;
\item $g$ is a strict Lyapunov function for $X_s$, that is, $g$
strictly decreases along every nonconstant orbit of $X_s$;
\item $\Inv\bigl(\overline B_\varepsilon,\varphi^s\bigr)=\{0\}
\subset B_\varepsilon=\operatorname{int}\overline B_\varepsilon$.
\end{enumerate}
In particular $\overline B_\varepsilon$ is a common isolating
neighborhood for the whole family, and $\{0\}$ is an isolated invariant
set of each $\varphi^s$.
\end{lemma}

\begin{proof}
\emph{(i)} By definition $\nabla^{h_s}g$ is the unique vector field with
$h_s(\nabla^{h_s}g,v)=dg(v)$ for every tangent vector $v$; since
$h_s$ is nondegenerate, $\nabla^{h_s}g(z)=0$ if and only if
$dg(z)=0$, and this last condition does not involve $h_s$.  By the
Cauchy--Riemann equations $dg=0$ iff $\df=0$, so
$\Crit(g)=\Crit(f)$ for every $s$; by hypothesis
$\Crit(g)\cap\overline B_\varepsilon=\{0\}$.

\emph{(ii)} Along an $X_s$-orbit $z(t)$,
\[
\frac{d}{dt}g(z(t))=dg\bigl(X_s(z(t))\bigr)
=-h_s\bigl(\nabla^{h_s}g,\nabla^{h_s}g\bigr)
=-\bigl\|\nabla^{h_s}g(z(t))\bigr\|_{h_s}^2\;\le\;0,
\]
with equality if and only if $z(t)\in\Crit(g)$, by (i).  Hence $g$
strictly decreases along every nonconstant orbit, so an orbit can never
return to a level of $g$ it has left.

\emph{(iii)} Fix $s$ and let $z:\R\to\overline B_\varepsilon$ be a complete
$X_s$-orbit with $z(\R)\subset\overline B_\varepsilon$.  Since $\overline B_\varepsilon$ is compact, the $\omega$-limit set $\omega(z)$ (the set of limit points of
$z(t)$ as $t\to+\infty$; $\alpha(z)$ is the same as $t\to-\infty$) and $\alpha(z)$ are nonempty, compact, connected and $X_s$-invariant (standard facts, see e.g.\ \cite[\S4.2]{Khalil2002}).
By (ii), $g(z(t))$ is nonincreasing and bounded, so it has limits $g_\pm$ as $t\to\pm\infty$, and by continuity of $g$ it equals $g_+$ on $\omega(z)$ and $g_-$ on $\alpha(z)$.  Let $w:\R\to\omega(z)$ be an orbit contained in $\omega(z)$.  Then $g(w(t))\equiv g_+$ is constant, so by (ii) $w$ is constant and lies in $\Crit(g)$.  Hence
$\omega(z)\subset\Crit(g)\cap\overline B_\varepsilon=\{0\}$, and since
$\omega(z)\neq\emptyset$, $\omega(z)=\{0\}$; likewise $\alpha(z)=\{0\}$.
Therefore $g_+=g_-=g(0)$.  If $z$ were nonconstant, $g(z(t))$ would be
strictly decreasing, which contradicts $\lim_{t\to-\infty}g(z(t))=g(0)=\lim_{t\to+\infty}g(z(t))$.
Hence $z(t)\equiv0$, proving
$\Inv(\overline B_\varepsilon,\varphi^s)=\{0\}$; as $s$ was arbitrary,
the claim holds for the whole family.
\end{proof}

\begin{corollary}[Independence of the metric]
\label{cor:metric}
The Conley index $h(\varphi_f,\{0\})$ does not depend on the choice of
Riemannian metric.
\end{corollary}

\begin{proof}
For metrics $h_0,h_1$ the pointwise combination $(1-s)h_0+sh_1$ is again
a smooth positive-definite metric, so they are joined by a continuous
path.  By \Cref{lem:isolation}, $\overline B_\varepsilon$ is an isolating
neighborhood for every flow of the path, with maximal invariant set
$\{0\}$, and \Cref{thm:continuation} applies.
\end{proof}

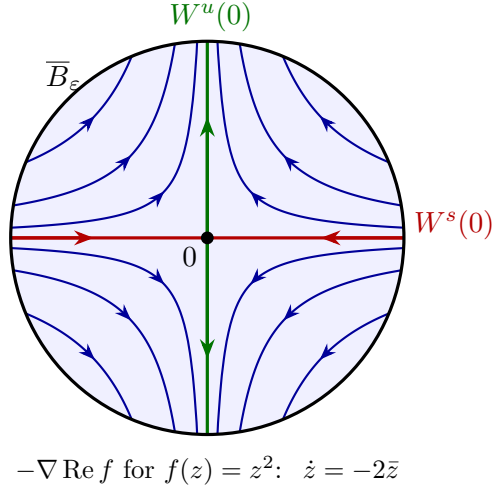
\begin{figure}[htb]
\centering
\begin{tikzpicture}[>=Stealth, scale=1.0]
\def\R{2.6}
\fill[blue!6] (0,0) circle (\R);
\begin{scope}
\clip (0,0) circle (\R);
\foreach \k in {0.35,1.1,2.4}{
  \foreach \sx/\sy in {1/1,-1/1,1/-1,-1/-1}{
    \draw[blue!60!black,thick,domain=0.03:6,samples=120,smooth,variable=\a]
      plot ({\sx*\a},{\sy*\k/\a});
  }
}
\end{scope}
\foreach \k in {0.35,1.1,2.4}{
  \pgfmathsetmacro{\m}{sqrt(\k)}
  \foreach \sx/\sy in {1/1,-1/1,1/-1,-1/-1}{
    \draw[blue!60!black,thick,-{Stealth[length=2.2mm]}]
      ({\sx*\m*1.04},{\sy*\m*0.96}) -- ({\sx*\m*0.96},{\sy*\m*1.04});
  }
}
\draw[red!70!black,very thick] (-\R,0) -- (\R,0);
\draw[red!70!black,very thick,-{Stealth[length=2.6mm]}] (-\R,0) -- (-1.5,0);
\draw[red!70!black,very thick,-{Stealth[length=2.6mm]}] (\R,0) -- (1.5,0);
\draw[green!45!black,very thick] (0,-\R) -- (0,\R);
\draw[green!45!black,very thick,-{Stealth[length=2.6mm]}] (0,0) -- (0,1.6);
\draw[green!45!black,very thick,-{Stealth[length=2.6mm]}] (0,0) -- (0,-1.6);
\draw[very thick] (0,0) circle (\R);
\fill (0,0) circle (2.4pt);
\node[below left] at (0,0) {$0$};
\node[red!70!black,right] at (\R,0.18) {$W^s(0)$};
\node[green!45!black,above] at (0.05,\R) {$W^u(0)$};
\node at (-1.9,2.15) {$\overline B_\varepsilon$};
\node[below] at (0,-\R-0.15) {\small $-\nabla\operatorname{Re}f$ for $f(z)=z^2$: $\ \dot z=-2\bar z$};
\end{tikzpicture}
\caption{Phase portrait of $-\nabla\re f$ for $f(z)=z^2$ in
$\overline B_\varepsilon$.  The orbits are the hyperbolas
$xy=\mathrm{const}$, and every nonconstant orbit leaves
$\overline B_\varepsilon$ in forward or backward time
(\Cref{lem:isolation}).}
\label{fig:isolation}
\end{figure}

\begin{remark}
The openness of isolation (\Cref{prop:openness}) is the flow-level counterpart of the stability of
Gromoll--Meyer pairs under $C^1$-small perturbations of the function
\cite[Lemma~5.2 and Thm.~5.6]{Chang1993}, which gives the constancy of
the critical groups when each function has a unique critical point in a
fixed open set.  For a morsification, where $\mu(f)$ critical points
appear, the sum property (\Cref{thm:sum}) is used instead; it does not
require the critical values to coincide.  The Marino--Prodi
perturbation \cite[Thm.~5.7]{Chang1993} resolves a degenerate critical
point into non-degenerate ones with indices in $[m,m+n]$ ($m$ the index,
$n$ the nullity); the holomorphic morsification of
\Cref{thm:morsification-count} produces exactly $\mu(f)$ points, all of
index $n$.
\end{remark}

A \emph{connecting orbit} from $q_i$ to $q_j$ is a complete orbit that
leaves $q_i$ and arrives at $q_j$, that is, $\alpha(z)=\{q_i\}$ and
$\omega(z)=\{q_j\}$.

\begin{lemma}[Connecting orbits and critical values]
\label{lem:connections}
Let $F$ be holomorphic on an open set $U\supset\overline B_\varepsilon$,
with finitely many critical points in $\overline B_\varepsilon$,
$q_1,\dots,q_r$, all in $B_\varepsilon$, and let $w_i=F(q_i)$.  Let $h$
be a Hermitian metric on $U$ and $X=-\nabla_h\re F$.  Then every complete
$X$-orbit $z$ contained in $\overline B_\varepsilon$ is either one of
the $q_i$ or satisfies $\alpha(z)=\{q_i\}$, $\omega(z)=\{q_j\}$ for
some $i\neq j$ with
\[
w_i-w_j\in(0,\infty).
\]
Hence if $w_i-w_j\notin(0,\infty)$ for all $i\neq j$, then
$\Inv(\overline B_\varepsilon,X)=\{q_1,\dots,q_r\}$.
\end{lemma}

\begin{proof}
As in \Cref{lem:isolation}, $\alpha(z)$ and $\omega(z)$ are nonempty,
compact, connected and invariant sets contained in the finite set
$\Crit(F)\cap\overline B_\varepsilon$ (since $\re F$ is constant on
each limit set, every orbit contained in it has $\frac{d}{dt}\re F=0$ and
so lies in $\Crit(F)$; the critical set is finite and limit sets are
connected), hence are single points
$q_i,q_j$ (in particular there are no periodic orbits, which close up, nor
homoclinic orbits, which leave a rest point and return to it,
since $\re F$ is strictly decreasing along nonconstant orbits).  By \Cref{rem:hermitian}, $\im F$ is constant along $z$, so
$\im w_i=\im w_j$, and $\re F$ is strictly decreasing along $z$ unless
$z$ is constant, so $\re w_i>\re w_j$ when $z$ is nonconstant.  If
$i=j$ this is impossible, so $z$ is constant.
\end{proof}

\begin{figure}[htb]
\centering
\begin{tikzpicture}[scale=0.8,transform shape,>=Stealth]
\def\R{3.0}
\fill[blue!6] (0,0) circle (\R);
\draw[very thick] (0,0) circle (\R);
\node at (-2.25,2.2) {$\overline B_\varepsilon$};
\coordinate (q1) at (-1.4,0.5);
\coordinate (q2) at (1.3,0.4);
\coordinate (q3) at (0.1,-1.6);
\begin{scope}
\clip (0,0) circle (\R);
\draw[blue!60!black,thick,-{Stealth[length=2.4mm]}] (-3.1,-1.2) .. controls (-2.2,-0.7) .. (q1);
\draw[blue!60!black,thick,-{Stealth[length=2.4mm]}] (q1) .. controls (-1.6,1.4) .. (-1.9,3.1);
\draw[blue!60!black,thick,-{Stealth[length=2.4mm]}] (2.9,1.7) .. controls (2.2,1.2) .. (q2);
\draw[blue!60!black,thick,-{Stealth[length=2.4mm]}] (q2) .. controls (1.7,-0.5) .. (2.5,-1.8);
\draw[blue!60!black,thick,-{Stealth[length=2.4mm]}] (-1.3,-3.1) .. controls (-0.7,-2.4) .. (q3);
\draw[blue!60!black,thick,-{Stealth[length=2.4mm]}] (q3) .. controls (1.0,-2.2) .. (1.5,-3.1);
\end{scope}
\draw[red!70!black,very thick,-{Stealth[length=3mm]}] (q1) .. controls (-0.4,0.95) and (0.4,0.9) .. (q2);
\foreach \q/\l/\pos in {q1/q_1/left,q2/q_2/right,q3/q_3/below}{
  \fill (\q) circle (2.4pt); \node[\pos] at (\q) {$\l$};
}
\node[red!70!black,above] at (0,1.0) {\small $q_1\to q_2$};
\node[align=center] at (0,-3.55)
  {\small a connection $q_i\to q_j$ inside $\overline B_\varepsilon$ requires\\[-1pt]
   \small $\operatorname{Im}w_i=\operatorname{Im}w_j$ and $\operatorname{Re}w_i>\operatorname{Re}w_j$};
\end{tikzpicture}
\caption{Complete orbits in $\overline B_\varepsilon$: rest points, and connections $q_i\to q_j$ only if $w_i-w_j>0$ (\Cref{lem:connections}).}
\label{fig:connections}
\end{figure}

\begin{lemma}[Rotation of the target]
\label{lem:rotation}
Let $F_0$ be holomorphic on an open set $U\supset\overline B_\varepsilon$
with finitely many critical points $q_1,\dots,q_r$ in
$\overline B_\varepsilon$, with critical values
$w_i=F_0(q_i)$.  For $\theta\in\R$ put $F_\theta=e^{-i\theta}F_0$.
Then $\Crit(F_\theta)=\Crit(F_0)$, nondegeneracy of the critical
points is preserved, and $F_\theta(q_i)-F_\theta(q_j)
=e^{-i\theta}(w_i-w_j)$.  Outside the finite set
$\{\arg(w_i-w_j)\bmod2\pi: w_i\neq w_j\}$ of angles, no difference
$F_\theta(q_i)-F_\theta(q_j)$ with $i\neq j$ is a positive real
number.  Hence such $\theta$ exist in every interval.
\end{lemma}

\begin{proof}
$e^{-i\theta}(w_i-w_j)\in(0,\infty)$ forces $w_i\ne w_j$ and
$\theta\equiv\arg(w_i-w_j)\pmod{2\pi}$: one forbidden angle for each
ordered pair.  Pairs with $w_i=w_j$ impose no condition, since the
difference is $0\notin(0,\infty)$.
\end{proof}

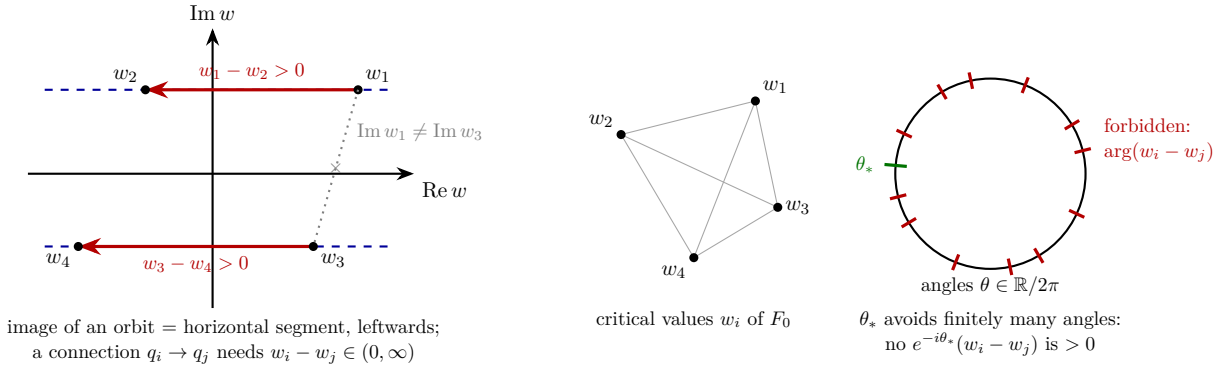
\begin{figure}[htb]
\centering
\begin{tikzpicture}[>=Stealth,scale=0.74,transform shape]
\begin{scope}
\draw[->,thick] (-3.3,0) -- (3.6,0) node[below right] {$\operatorname{Re}w$};
\draw[->,thick] (0,-2.4) -- (0,2.6) node[above] {$\operatorname{Im}w$};
\coordinate (a1) at (2.6,1.5);
\coordinate (a2) at (-1.2,1.5);
\coordinate (a3) at (1.8,-1.3);
\coordinate (a4) at (-2.4,-1.3);
\draw[blue!60!black,thick,dashed] (-3.0,1.5) -- (3.2,1.5);
\draw[blue!60!black,thick,dashed] (-3.0,-1.3) -- (3.2,-1.3);
\draw[red!70!black,very thick,-{Stealth[length=3mm]}] (a1) -- (a2);
\draw[red!70!black,very thick,-{Stealth[length=3mm]}] (a3) -- (a4);
\foreach \p/\l/\pos in {a1/w_1/above right,a2/w_2/above left,a3/w_3/below right,a4/w_4/below left}{
  \fill (\p) circle (2.3pt); \node[\pos] at (\p) {$\l$};
}
\node[red!70!black,above] at (0.7,1.55) {\small $w_1-w_2>0$};
\node[red!70!black,below] at (-0.3,-1.35) {\small $w_3-w_4>0$};
\draw[gray,dotted,thick] (a1) -- (a3);
\node[gray] at (2.2,0.1) {\small $\times$};
\node[gray,anchor=west] at (2.45,0.75) {\small $\operatorname{Im}w_1\neq\operatorname{Im}w_3$};
\node[align=center] at (0.2,-3.0) {\small image of an orbit $=$ horizontal segment, leftwards;\\[-1pt]
\small a connection $q_i\to q_j$ needs $w_i-w_j\in(0,\infty)$};
\end{scope}
\begin{scope}[xshift=9.2cm]
\coordinate (b1) at (0.5,1.3);
\coordinate (b2) at (-1.9,0.7);
\coordinate (b3) at (0.9,-0.6);
\coordinate (b4) at (-0.6,-1.5);
\foreach \i/\j in {1/2,1/3,1/4,2/3,2/4,3/4}{
  \draw[gray!70] (b\i) -- (b\j);
}
\foreach \p/\l/\pos in {b1/w_1/above right,b2/w_2/above left,b3/w_3/right,b4/w_4/below left}{
  \fill (\p) circle (2.3pt); \node[\pos] at (\p) {$\l$};
}
\node at (-0.6,-2.6) {\small critical values $w_i$ of $F_0$};
\def\C{(4.7,0)}
\draw[thick] (4.7,0) circle (1.7);
\node at (4.7,-2.0) {\small angles $\theta\in\mathbb R/2\pi$};
\foreach \i/\j in {1/2,1/3,1/4,2/3,2/4,3/4}{
  \pgfmathanglebetweenpoints{\pgfpointanchor{b\j}{center}}{\pgfpointanchor{b\i}{center}}
  \let\ang\pgfmathresult
  \draw[red!70!black,very thick] ($(4.7,0)+(\ang:1.55)$) -- ($(4.7,0)+(\ang:1.85)$);
  \pgfmathsetmacro{\angb}{\ang+180}
  \draw[red!70!black,very thick] ($(4.7,0)+(\angb:1.55)$) -- ($(4.7,0)+(\angb:1.85)$);
}
\draw[green!45!black,very thick] ($(4.7,0)+(175:1.5)$) -- ($(4.7,0)+(175:1.9)$);
\node[green!45!black,left] at ($(4.7,0)+(175:1.95)$) {\small $\theta_*$};
\node[red!70!black,align=left,right] at (6.6,0.6) {\small forbidden:\\[-1pt]\small $\arg(w_i-w_j)$};
\node[align=center] at (4.7,-2.85) {\small $\theta_*$ avoids finitely many angles:\\[-1pt]
\small no $e^{-i\theta_*}(w_i-w_j)$ is $>0$};
\end{scope}
\end{tikzpicture}
\caption{Left: a connection requires $w_i-w_j\in(0,\infty)$ (\Cref{lem:connections}).  Right: forbidden angles $\arg(w_i-w_j)$ and the chosen $\theta_*$ (\Cref{lem:rotation}).}
\label{fig:connections-rotation}
\end{figure}

\subsection{Proof of Theorem A}
\label{sec:proofA}

We continue $f$ to a morsification with $\mu(f)$ nondegenerate
critical points (\Cref{fig:proofA}), rotate the target so that the
first integral $\im f_t$ forbids connecting orbits, and apply the sum
property.

\begin{figure}[htb]
\centering
\begin{tikzpicture}[>=Stealth, scale=0.85]
\draw[thick] (0,0) circle (1.0);
\filldraw[red!70!black] (0,0) circle (2.2pt);
\node at (0,-1.35) {\scriptsize degenerate, $\mu(f)$};
\node at (0,1.35) {$f$};

\draw[->,thick] (1.4,0) -- (2.6,0);
\node[above] at (2.0,0.05) {\scriptsize morsify};

\draw[thick] (4.5,0) circle (1.0);
\foreach \pt in {(4.15,0.35),(4.75,0.25),(4.35,-0.35),(4.85,-0.3)}{
  \filldraw[blue!55!black] \pt circle (1.8pt);
}
\node at (4.5,-1.35) {\scriptsize $\mu(f)$ Morse pts., index $n$};
\node at (4.5,1.35) {$f_t$};

\draw[->,thick] (5.9,0) -- (7.1,0);
\node[above] at (6.5,0.05) {\scriptsize sum property};

\foreach \ang in {20,80,140,200}{
  \draw[thick] ($(9.3,0)+(\ang:0.5)$) circle (0.5);
}
\filldraw (9.3,0) circle (1.5pt);
\node at (9.3,-1.35) {$\bigvee_{\mu(f)} S^{\,n}$};
\end{tikzpicture}
\caption{Continuation to $\mu(f)$ hyperbolic rest points of index $n$ ($\mu(f)=4$ drawn), then the sum property.}
\label{fig:proofA}
\end{figure}
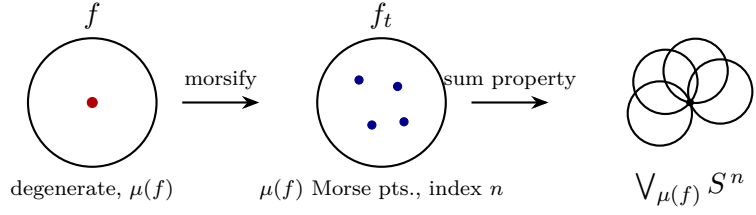

\begin{figure}[htb]
\centering
\begin{tikzpicture}[scale=0.7,transform shape,>=Stealth]
\foreach \x/\lab in {0/{$t=0$: degenerate, $\mu(f)$},4.2/{$t$ small},8.4/{$t=t_*$: $p_i(t_*)\in B_{\varepsilon/2}$}}{
 \begin{scope}[xshift=\x cm]
 \fill[blue!6] (0,0) circle (1.5);
 \draw[very thick] (0,0) circle (1.5);
 \draw[dashed,gray] (0,0) circle (0.75);
 \node at (-1.05,1.15) {\small $N_0$};
 \node at (0,-2.05) {\small \lab};
 \end{scope}}
\fill[red!70!black] (0,0) circle (2.6pt);
\foreach \x/\y in {-0.12/0.1,0.12/0.12,-0.1/-0.12,0.14/-0.1}{\fill[blue!55!black] (4.2+\x,\y) circle (1.8pt);}
\foreach \x/\y/\u/\v in {-0.12/0.1/-0.45/0.35,0.12/0.12/0.4/0.3,-0.1/-0.12/-0.3/-0.4,0.14/-0.1/0.45/-0.28}{
 \draw[->,gray] (8.4+\x,\y) -- (8.4+\u,\v); \fill[blue!55!black] (8.4+\u,\v) circle (1.8pt);}
\end{tikzpicture}
\caption{The critical point $0$ of $f$ splits into $\mu(f)$ Morse points of $f_t$ inside $N_0$.}
\label{fig:morsify-ball}
\end{figure}
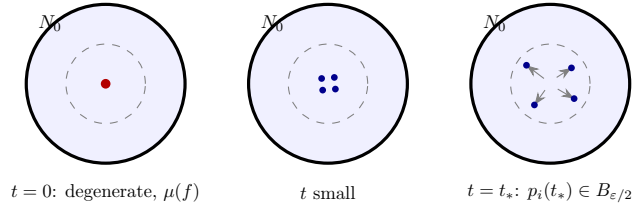

Fix $\varepsilon>0$ with
$\Crit(f)\cap\overline B_\varepsilon=\{0\}$, and write
$N_0=\overline B_\varepsilon$.  Let $h_0$ denote a Hermitian metric
on a neighborhood of $N_0$.  An arbitrary Riemannian metric is joined to the standard Hermitian
metric by convex interpolation (Hermitian metrics are needed from this
point on, since the first integral $\im f$ requires $J$-invariance); by \Cref{cor:metric} the Conley index is unchanged, so we work with
the Hermitian metric $h_0$.

\medskip\noindent\textbf{Step 1: continuation to a rotated morsification.}
For each $\theta\in[0,2\pi]$ the germ $e^{-i\theta}f$ has the same
critical set as $f$, so \Cref{lem:isolation} applied to
$e^{-i\theta}f$ shows that $N_0$ is an isolating neighborhood, with
$\Inv=\{0\}$, for the flow of
$X_{\theta,0}=-\nabla^{h_0}\re(e^{-i\theta}f)$.  Let $f_t$ be a
morsification of $f$ (\Cref{def:morsification}) and put
$X_{\theta,t}=-\nabla^{h_0}\re(e^{-i\theta}f_t)$.  The fields $X_{\theta,t}$ depend smoothly on $(\theta,t)$, so by the continuous dependence of solutions of ordinary differential equations on parameters the flows depend continuously on $(\theta,t)$.  By \Cref{prop:openness} the set of parameters $(\theta,t)$ for which $N_0$ is isolating is open.  It contains the compact set $[0,2\pi]\times\{0\}$, hence it contains $[0,2\pi]\times\{|t|\le t_0\}$ for some $t_0>0$.
Shrinking $t_0$, \Cref{thm:morsification-count} gives, for
$0<|t|\le t_0$, exactly $\mu:=\mu(f)$ nondegenerate critical points
$p_1(t),\dots,p_\mu(t)\in B_{\varepsilon/2}$ of $f_t$ in $N_0$.

Since $t_0$ does not depend on $\theta$, we may fix $t_*$ with
$0<|t_*|\le t_0$ and put $w_i=f_{t_*}(p_i(t_*))$.  By
\Cref{lem:rotation} we then fix $\theta_*$ such that no difference
$e^{-i\theta_*}(w_i-w_j)$, $i\ne j$, is a positive real number, and
set
\[
F=e^{-i\theta_*}f_{t_*},\qquad X_F=-\nabla^{h_0}\re F,\qquad
\varphi_F=\text{its flow}.
\]
The parameter path $(s\theta_*,0)$, $s\in[0,1]$, followed by
$(\theta_*,rt_*)$, $r\in[0,1]$, lies in the region where $N_0$ is
isolating (\Cref{fig:param}), so continuation gives
\begin{equation}\label{eq:cont}
h\bigl(\varphi_f,\{0\}\bigr)
\;\simeq\;
h\bigl(\varphi_F,\,S\bigr),
\qquad
S=\Inv\bigl(N_0,\varphi_F\bigr).
\end{equation}

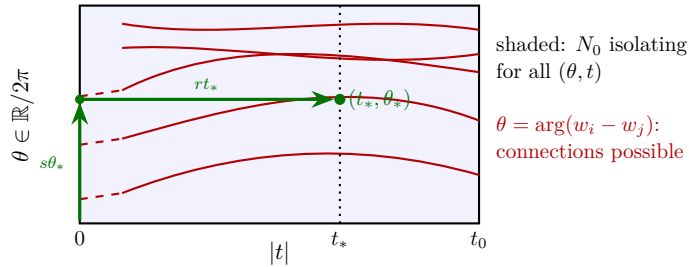
\begin{figure}[htb]
\centering
\begin{tikzpicture}[scale=0.8,transform shape,>=Stealth]

\fill[blue!5] (0,0) rectangle (6.6,3.6);
\draw[thick] (0,0) rectangle (6.6,3.6);
\node[rotate=90] at (-0.95,1.8) {$\theta\in\mathbb R/2\pi$};
\node at (3.3,-0.5) {$|t|$};
\node[below] at (0,0) {\small $0$};
\node[below] at (6.6,0) {\small $t_0$};
\foreach \p/\a/\b in {0.5/0.7/0.9,1.4/0.5/1.1,2.2/0.9/0.6}{
 \draw[red!70!black,thick] (0.7,\p) .. controls (2.5,{\p+\a}) and (4.5,{\p+\b}) .. (6.6,{\p+0.3});
 \draw[red!70!black,thick,dashed] (0,\p-0.1) -- (0.7,\p);
}
\foreach \p/\a/\b in {2.9/0.1/-0.4,3.3/-0.3/0.2}{
 \draw[red!70!black,thick] (0.7,\p) .. controls (2.5,{\p+\a}) and (4.5,{\p+\b}) .. (6.6,{\p-0.1});
}
\def\ts{4.3}
\draw[dotted,thick] (\ts,0)--(\ts,3.6);
\node[below] at (\ts,0) {\small $t_*$};
\fill[green!45!black] (\ts,2.05) circle (2.6pt);
\node[green!45!black,right] at (\ts,2.05) {\small $(t_*,\theta_*)$};
\draw[->,green!45!black,very thick] (0,0.05)--(0,2.0);
\draw[->,green!45!black,very thick] (0,2.05)--(\ts-0.1,2.05);
\fill[green!45!black] (0,2.05) circle (2.2pt);
\node[green!45!black,left] at (-0.1,1.0) {\scriptsize $s\theta_*$};
\node[green!45!black,above] at (2.1,2.05) {\scriptsize $rt_*$};
\node[red!70!black,align=left,right] at (6.6+0.15,1.4) {\small $\theta=\arg(w_i-w_j)$:\\[-1pt]\small connections possible};
\node[align=left,right] at (6.6+0.15,2.7) {\small shaded: $N_0$ isolating\\[-1pt]\small for all $(\theta,t)$};
\end{tikzpicture}
\caption{Parameter plane: $N_0$ is isolating everywhere; the path of Step~1 in green.}
\label{fig:param}
\end{figure}

\medskip\noindent\textbf{Step 2: the invariant set of the morsified
flow.}
The critical points of $F$ in $N_0$ are the $p_i(t_*)$, all in
$B_{\varepsilon/2}$, with values $e^{-i\theta_*}w_i$.  By
\Cref{lem:connections} applied to $F$, every complete orbit of
$\varphi_F$ in $N_0$ is a rest point or a connecting orbit
$p_i\to p_j$, $i\ne j$, with $e^{-i\theta_*}(w_i-w_j)\in(0,\infty)$.
By the choice of $\theta_*$ the latter do not exist, hence
\[
S=\{p_1(t_*)\}\sqcup\dots\sqcup\{p_\mu(t_*)\},
\]
a disjoint union of isolated invariant sets.

\medskip\noindent\textbf{Step 3: the rest points are hyperbolic of
index $n$.}
Each $p_j(t_*)$ is a nondegenerate critical point of the holomorphic
function $F$, hence of $\re F$ with Morse index exactly $n$
(\Cref{lem:index-n}).  The linearization of $X_F$ at $p_j$ is $-A$, where $A$ is the
$h_0$-self-adjoint operator with
$h_0(Au,v)=\operatorname{Hess}_{p_j}(\re F)(u,v)$.  In an
$h_0$-orthonormal basis $A$ is the symmetric matrix of the Hessian,
which has signature $(n,n)$ by \Cref{lem:index-n}.  Hence $A$ has $n$
positive and $n$ negative eigenvalues, so $DX_F=-A$ has no eigenvalue on
the imaginary axis and exactly $n$ eigenvalues with positive real part.
Thus $p_j$ is hyperbolic with unstable dimension $n$.

\medskip\noindent\textbf{Step 4: sum property and conclusion.}
By \Cref{thm:sum,thm:hyperbolic},
\[
h\bigl(\varphi_F,S\bigr)
\;\simeq\;
\bigvee_{j=1}^{\mu}h\bigl(\varphi_F,\{p_j(t_*)\}\bigr)
\;\simeq\;
\bigvee_{\mu}S^{\,n}.
\]
Combining \eqref{eq:cont} and \Cref{thm:wedge},
\[
h\bigl(\varphi_f,\{0\}\bigr)
\;\simeq\;
\bigvee_{\mu(f)}S^{\,n}
\;\simeq\;
\Sigma\Bigl(\bigvee_{\mu(f)}S^{\,n-1}\Bigr)
\;\simeq\;
\Sigma F_f .
\]
Here suspension commutes with finite wedges; for $n=1$,
$F_f$ has $\mu(f)+1$ points and $\Sigma F_f\simeq\bigvee_{\mu(f)}S^1$.
The homological statements follow.  \qed

\begin{remark}
\label{rem:angles}
By \Cref{lem:connections}, a connecting orbit of $\re(e^{-i\theta}f_t)$
between $p_i$ and $p_j$ can occur only for
$\theta\equiv\arg(f_t(p_i)-f_t(p_j))$.  These finitely many
exceptional angles are avoided in the proof.  Since the Conley index is
invariant under continuation, any admissible angle gives the same
index.
\end{remark}

The following proposition is independent of the proof of Theorem~A; it
gives a complementary geometric description of the index of $\{0\}$
through the exit set $V^+$.

\begin{proposition}[Homogeneous leading part: Conley's computation]
\label{prop:homog}
Let $f=f_k+f'$, where $f_k$ is a homogeneous polynomial of degree
$k\ge2$ with an isolated critical point at $0$ and $f'$ is holomorphic
with $f'=O(|z|^{k+1})$.  Let $\varphi_f$ be the negative gradient flow
of $\re f$ for the Euclidean metric, and put
$V^+=\{z\in S^{2n-1}:\re f_k(z)\le0\}\subset S^{2n-1}=\partial B_1$.
Then $V^+\neq\emptyset$, $V^+\neq S^{2n-1}$, and
\[
h(\varphi_f,\{0\})\;\simeq\;\Sigma V^+\;\simeq\;\Sigma F_f .
\]
\end{proposition}

\begin{proof}
\emph{Step 1: reduction to the homogeneous flow.}
For $\lambda\in(0,1]$ put $f_\lambda(z)=\lambda^{-k}f(\lambda z)$ and
$f_0=f_k$.  Writing $f'=\sum_{j>k}f_j$ with $f_j$ homogeneous of degree $j$,
$f_\lambda=f_k+\sum_{j>k}\lambda^{j-k}f_j$, so $(\lambda,z)\mapsto
f_\lambda(z)$ is smooth on $[0,1]\times\overline B_1$ and $\nabla\re
f_\lambda\to\nabla\re f_k$ in $C^0$ (indeed $C^\infty$) as
$\lambda\to0$.  Since $\nabla f_\lambda(z)=\lambda^{1-k}(\nabla
f)(\lambda z)$, the dilation $z\mapsto\lambda z$ carries orbits of
$-\nabla\re f_\lambda$ to orbits of $-\nabla\re f$, with time multiplied
by the constant $\lambda^{2-k}>0$; hence $\overline B_1$ is isolating for
$f_\lambda$ iff $\overline B_\lambda$ is isolating for $f$ (for the
corresponding negative gradient flows), with the
same index of $\{0\}$.  For small $\lambda$ this holds by
\Cref{lem:isolation}.  For $\lambda=0$ it holds because $f_k$ is
holomorphic with $\Crit(f_k)=\{0\}$ and \Cref{lem:isolation} applies to
$f_k$ on $\overline B_1$.  By continuation (\Cref{thm:continuation}) over
$\lambda\in[0,\lambda_0]$, 
$h(\varphi_f,\{0\})\simeq h(\varphi_{f_k},\{0\})$.  Conley's block is applied to $\varphi_{f_k}$ and not to $\varphi_f$,
since on $\{\re f_k=0\}\cap S^{2n-1}$ the radial component of the flow
of $f_k$ vanishes to first order, and the remainder $-\nabla\re f'$ can
change the sign of the radial derivative.

\emph{Step 2: Conley's computation for $f_k$.}
The flow of $f_k$ is $\dot z=\nabla V(z)$ with $V=-\re f_k$, real
homogeneous of degree $k$ with $\Crit(V)=\{0\}$.
By Conley \cite[Ch.~IV, \S1.5A]{Conley1978} (see also Ch.~III,
\S6.2A, where the unit ball $B_1$ is shown to be an isolating block
whose exit set is $\{V\ge0\}\cap\partial B_1=V^+$), the origin is an
isolated invariant set and $h(\varphi_{f_k},\{0\})=\Sigma^1\wedge[V^+/v^+]$ in Conley's
notation, that is, the reduced suspension $\Sigma V^+$ of $V^+$ pointed
at $v^+$, provided $V^+\neq\emptyset$ and
$V^+\neq\partial B_1$.  Both hold because
$\re f_k(e^{i\pi/k}z)=-\re f_k(z)$ and $f_k\not\equiv0$.

\emph{Step 3: identification of $V^+$.}
By Milnor \cite{Milnor1968}, $f_k/|f_k|$
fibers $S^{2n-1}\setminus K$ over $S^1$, $K=f_k^{-1}(0)\cap S^{2n-1}$,
with fiber $\{z\in S^{2n-1}:f_k(z)>0\}\cong f_k^{-1}(1)$, the Milnor
fiber of $f_k$, and this is an open book with binding $K$.  The set
$V^+$ is the binding together with the pages over a closed half circle,
so the Milnor fibration, trivial over the contractible arc, makes $V^+$
a collar thickening of any closed page (a page together with the
binding); hence $V^+$ deformation retracts onto a closed page.  The
open page $\{f_k>0\}\cap S^{2n-1}$ is diffeomorphic to $f_k^{-1}(1)$ via
$z\mapsto f_k(z)^{-1/k}z$, with the positive real $k$th root since
$f_k(z)>0$, by homogeneity, and
the closed page is homotopy equivalent to it (a homotopy equivalence,
not an identification of manifolds, since the closed page has the
binding as boundary).  Finally
$f_k^{-1}(1)\simeq F_f$: by homogeneity and isolatedness there is
$c>0$ with $\|\partial f_k(z)\|\ge c|z|^{k-1}$, while
$\|\partial f'(z)\|=O(|z|^k)$, so $\partial f_k+s\,\partial f'$,
$s\in[0,1]$, has no zeros on a small sphere, the degrees of
$\partial f$ and $\partial f_k$ there agree, and
$\mu(f)=\mu(f_k)$ by
\cite[Ch.~I, \S5.2, Thm.~1]{ArnoldGuseinZadeVarchenko1985}; by
Milnor, $F_f\simeq\bigvee_{\mu(f)}S^{n-1}\simeq
\bigvee_{\mu(f_k)}S^{n-1}\simeq f_k^{-1}(1)$.

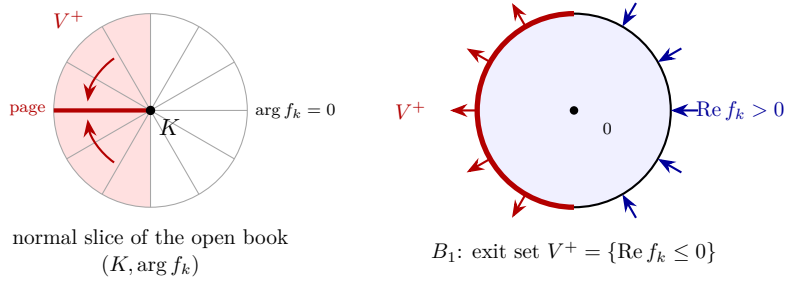
\begin{figure}[htb]
\centering
\begin{tikzpicture}[scale=0.8,transform shape,>=Stealth]
\begin{scope}
 \fill[red!12] (0,0) -- (90:1.6) arc (90:270:1.6) -- cycle;
 \foreach \a in {0,30,...,330}{\draw[gray!70] (0,0)--(\a:1.6);}
 \draw[gray!70] (0,0) circle (1.6);
 \draw[red!70!black,line width=1.6pt] (0,0)--(180:1.6);
 \draw[->,red!70!black,thick] (125:1.05) arc (125:172:1.05);
 \draw[->,red!70!black,thick] (235:1.05) arc (235:188:1.05);
 \fill (0,0) circle (2.2pt);
 \node[below right] at (0,0) {$K$};
 \node[red!70!black,left] at (-1.6,0) {\scriptsize page};
 \node[right] at (1.6,0) {\scriptsize $\arg f_k=0$};
 \node[red!70!black] at (-1.35,1.5) {\small $V^+$};
 \node[align=center] at (0,-2.35) {\small normal slice of the open book\\[-1pt]\small $(K,\arg f_k)$};
\end{scope}
\begin{scope}[xshift=7.0cm]
 \fill[blue!6] (0,0) circle (1.6);
 \draw[thick] (0,0) circle (1.6);
 \draw[red!70!black,line width=2pt] (90:1.6) arc (90:270:1.6);
 \foreach \a in {120,150,180,210,240}{\draw[->,red!70!black,thick] (\a:1.6)--(\a:2.05);}
 \foreach \a in {-60,-30,0,30,60}{\draw[->,blue!60!black,thick] (\a:2.05)--(\a:1.6);}
 \fill (0,0) circle (2pt);
 \node at (0.55,-0.3) {\scriptsize $0$};
 \node[red!70!black] at (-2.7,0) {\small $V^+$};
 \node[blue!60!black] at (2.75,0) {\small $\operatorname{Re}f_k>0$};
 \node at (0,-2.35) {\small $B_1$: exit set $V^+=\{\operatorname{Re}f_k\le0\}$};
\end{scope}
\end{tikzpicture}
\caption{Open book of $f_k$ (left) and the block $B_1$ with exit set $V^+$ (right).}
\label{fig:openbook}
\end{figure}

\begin{figure}[htb]
\centering
\begin{tikzpicture}[scale=0.8,transform shape,>=Stealth]
\begin{scope}
 \fill[blue!6] (0,0) circle (1.2);
 \draw[thick] (0,0) circle (1.2);
 \draw[red!70!black,line width=2pt] (150:1.2) arc (150:210:1.2);
 \node[red!70!black] at (-1.65,0) {\small $V^+$};
 \node at (0.1,0) {$B_1$};
 \node at (0,-1.9) {\small $N_1/N_2=B_1/V^+$};
\end{scope}
\node at (2.3,0) {$\simeq$};
\begin{scope}[xshift=4.6cm]
 \fill[blue!6] (-1,-0.9)--(1,-0.9)--(0,1.1)--cycle;
 \draw[thick] (-1,-0.9)--(0,1.1)--(1,-0.9);
 \draw[red!70!black,line width=2pt] (-1,-0.9)--(1,-0.9);
 \fill (0,1.1) circle (2pt);
 \node[red!70!black,below] at (0,-0.95) {\small $V^+$};
 \node at (0,-1.9) {\small $CV^+/V^+$};
\end{scope}
\node at (6.9,0) {$=$};
\begin{scope}[xshift=9.2cm]
 \fill[blue!6] (0,1.1) .. controls (1.3,0.3) and (1.3,-0.5) .. (0,-0.9)
                       .. controls (-1.3,-0.5) and (-1.3,0.3) .. (0,1.1);
 \draw[thick] (0,1.1) .. controls (1.3,0.3) and (1.3,-0.5) .. (0,-0.9)
                       .. controls (-1.3,-0.5) and (-1.3,0.3) .. (0,1.1);
 \fill (0,1.1) circle (2pt);
 \fill[red!70!black] (0,-0.9) circle (2.6pt);
 \node[red!70!black,below] at (0,-0.95) {\small $[V^+]$};
 \node at (0,-1.9) {\small $\Sigma V^+\simeq\Sigma F_f$};
\end{scope}
\end{tikzpicture}
\caption{$B_1/V^+$ is the reduced suspension of $V^+\simeq F_f$.}
\label{fig:reduced-susp}
\end{figure}
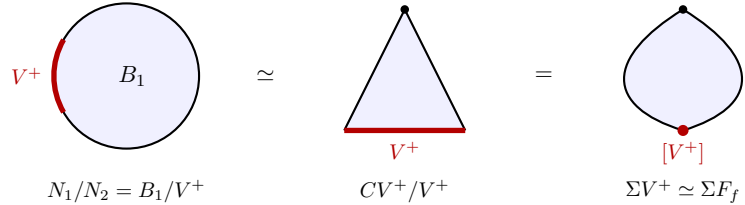

\end{proof}

\begin{remark}
\label{rem:conley-homog}
For homogeneous $f_k$ this is a one-step computation in Conley's
framework, with $\mu(f)=(k-1)^n$ spheres.  For a singularity such as $A_k$,
$k\ge2$, the leading homogeneous part has no isolated critical point; the ball is then not an isolating
block, and the morsification argument above provides the computation.
\end{remark}

\begin{proposition}[Reduction to the corank]
\label{prop:splitting}
Let $f:(\C^n,0)\to(\C,0)$ be a holomorphic germ with an isolated critical
point at $0$, and let $m$ be the corank of its complex Hessian at $0$.
Then in suitable holomorphic coordinates $(u,v)\in\C^m\times\C^{n-m}$,
\[
f(u,v)=g(u)+v_1^2+\dots+v_{n-m}^2 ,
\]
where $g:(\C^m,0)\to(\C,0)$ has an isolated critical point at $0$ and
vanishing $2$-jet, and
\[
h(\varphi_f,\{0\})\;\simeq\;\Sigma^{\,n-m}\,h(\varphi_g,\{0\}),
\]
where $\varphi_f,\varphi_g$ are the negative gradient flows of $\re f$ and
$\re g$.  In particular, if $h(\varphi_g,\{0\})\simeq\Sigma V^+$ as in
\Cref{prop:homog}, then $h(\varphi_f,\{0\})\simeq\Sigma^{\,n-m+1}V^+$.
\end{proposition}

\begin{proof}
Such coordinates exist by the holomorphic splitting lemma (the
holomorphic counterpart of the Gromoll--Meyer splitting lemma
\cite{GromollMeyer1969}, obtained from the holomorphic implicit function
theorem and the Morse lemma with parameters); the function $g$ has
vanishing $2$-jet because $m$ is the corank of the Hessian, and an
isolated critical point because $f$ has one.  Let $\psi$ be the
biholomorphism to these coordinates.  Transporting the flow by $\psi$
gives the negative gradient flow of $\re(f\circ\psi)$ for the
pulled-back metric, again a Riemannian metric, and conjugate flows have
the same index; by \Cref{cor:metric} the index does not depend on the
metric, so we may use the product metric in the coordinates $(u,v)$.  The
flow is then the product of the flow of $-\nabla\re g$ on $\C^m$ and of
the flow of $-\nabla\re\sum_jv_j^2$ on $\C^{n-m}$.  The origin is a
hyperbolic rest point of the second flow, whose real unstable dimension is
$n-m$: writing $v_j=x_j+iy_j$, $-\nabla\re v_j^2=(-2x_j,\,2y_j)$, so each
$v_j$ contributes one stable and one unstable real direction.  Its
index is therefore $S^{n-m}$ by \Cref{thm:hyperbolic}.  If $(N_1,N_2)$ and
$(M_1,M_2)$ are index pairs for the factors, then
$(N_1\times M_1,\;N_1\times M_2\cup N_2\times M_1)$ is an index pair for
the product flow, with quotient $N_1/N_2\wedge M_1/M_2$
\cite[Ch.~III, \S6.1D]{Conley1978}; hence the index of a product of
isolated invariant sets is the smash product of the indices.  Therefore
$h(\varphi_f,\{0\})\simeq S^{n-m}\wedge h(\varphi_g,\{0\})
=\Sigma^{n-m}h(\varphi_g,\{0\})$.
\end{proof}

\begin{example}[The $A_k$ singularity]
For $f=z_1^{k+1}+z_2^2+\dots+z_n^2$ one has $m=1$ and $g=z^{k+1}$ on
$\C$, whose leading part is homogeneous with an isolated critical point.
Here $V^+=\{\re z^{k+1}\le0\}\cap S^1$ is a disjoint union of $k+1$
closed arcs, hence homotopy equivalent to $k+1$ points, so
$\Sigma V^+\simeq\bigvee_kS^1$ by \Cref{prop:homog}, and
\Cref{prop:splitting} gives
$h(\varphi_f,\{0\})\simeq\Sigma^{n-1}\bigvee_kS^1\simeq\bigvee_kS^n$,
in agreement with \Cref{thm:wedge} and Theorem~A.  In particular
$\CH_j(\varphi_f,\{0\};\Z)\cong\Z^k$ for $j=n$ and $0$ otherwise.
\end{example}

\subsection{Homological consequences and the Poincar\'e--Hopf index}
\label{sec:homology}

\begin{corollary}
\label{cor:CH}
Under the hypotheses of Theorem~A,
\[
\CH_k\bigl(\varphi_f,\{0\};\Z\bigr)
\cong
\wtH_k\bigl(\Sigma F_f;\Z\bigr)
\cong
\wtH_{k-1}\bigl(F_f;\Z\bigr)
=
\begin{cases}
\Z^{\mu(f)}, & k=n,\\
0,& k\neq n,
\end{cases}
\]
and $\mu(f)=\rank\CH_n(\varphi_f,\{0\};\Z)$.
\end{corollary}

The isomorphism $\CH_n(\varphi_f,\{0\};\Z)\cong\wtH_{n-1}(F_f;\Z)$
depends on the choice of homotopy equivalence $h(\varphi_f,\{0\})\simeq\Sigma F_f$.
It identifies the underlying abelian groups, but not the intersection
form, the Seifert form or the monodromy.

\begin{corollary}[Compatibility with Poincar\'e--Hopf]
\label{cor:PH}
Under the hypotheses of Theorem~A,
\[
\widetilde\chi\bigl(h(\varphi_f,\{0\})\bigr)
=(-1)^n\,\mu(f),
\qquad
\ind_{\mathrm{PH}}\bigl(-\nabla\re f,0\bigr)=(-1)^n\mu(f).
\]
\end{corollary}

\begin{proof}
The first equality follows from Theorem~A, since
$\widetilde\chi(S^n)=(-1)^n$ and $h(\varphi_f,\{0\})\simeq
\bigvee_{\mu(f)}S^n$.

For the second, we use McCord's theorem \cite{McCord1989}; see also \cite{RazvanFotouhi2008}, where the Euler characteristic of the Conley index is taken as the definition of the Poincar\'e index of an isolated invariant set.  In its
normalization, for a smooth vector field $\mathscr X$ on a manifold of
real dimension $m$, with $S$ an isolated invariant set all of whose
zeros are nondegenerate,
$(-1)^m\chi(M;S)=\sum_{x\in S}\iota_x(\mathscr X)$
\cite[Theorem]{McCord1989}, where $\chi(M;S)$ is the Euler
characteristic of the homological Conley index, equal to our
$\widetilde\chi(h(\varphi,S))$, and $\iota_x(\mathscr X)$ is the
ordinary index $\operatorname{sign}\det(D\mathscr X_x)$; explicitly
$\chi(M;S)=\sum_k(-1)^k\operatorname{rank}\CH_k(\varphi,S)$.  Here
$m=2n$, so $(-1)^m=1$.  Apply this to the flow $\varphi_F$ of
the proof of Theorem~A and to $S$ (Step~2): its zeros are the
$\mu(f)$ nondegenerate points $p_j(t_*)$, and the linearization
$DX_F=-A$ of Step~3 has $n$ positive and $n$ negative eigenvalues
(the Hessian $A$ has signature $(n,n)$ by \Cref{lem:index-n}), so $\iota_{p_j(t_*)}(X_F)=(-1)^n$ for every $j$.
Hence $\chi(M;S)=\mu(f)(-1)^n$.  By continuation
(\Cref{thm:continuation}), $\chi(M;\cdot)$ agrees with its value at
the original degenerate point, so
$\widetilde\chi(h(\varphi_f,\{0\}))=(-1)^n\mu(f)$ equals the
Poincar\'e--Hopf index of $-\nabla\re f$ at $0$.  The same value follows from the local degree: the degree of $\df$ at $0$ equals $\mu(f)$ \cite[Ch.~I, \S5.2, Thm.~1]{ArnoldGuseinZadeVarchenko1985}, complex conjugation acts on $S^{2n-1}\subset\C^n$ with degree $(-1)^n$, and the antipodal sign in real dimension $2n$ is $+1$, so $\ind_{\mathrm{PH}}(-\overline{\df},0)=(-1)^n\mu(f)$.
\end{proof}

\begin{remark}
The index $\ind_{\mathrm{PH}}(-\nabla\re f,0)$ is also computable
algebraically \cite{EisenbudLevine1977,Khimshiashvili1977}; see
\cite{BrasseletSeadeSuwa2009,EbelingGuseinZadeSeade2004} for indices of
vector fields and $1$-forms on singular varieties.
\end{remark}

\section{Examples}
\label{sec:examples}

\begin{example}[The Morse singularity]
For $f=z_1^2+\dots+z_n^2$, the origin is already a nondegenerate
critical point of $f$, so no morsification is needed.  By
\Cref{thm:wedge}, $\mu(f)=1$ and $F_f\simeq S^{n-1}$, so
$\Sigma F_f\simeq S^n$.

On the dynamical side, $\re f=\sum_j(x_j^2-y_j^2)$ and
$-\nabla\re f(z)=(-2x_1,2y_1,\dots,-2x_n,2y_n)$, which under
$\C^n\cong\R^{2n}$ is $-2\overline z$: the linearization at
$0$ has $n$ contracting and $n$ expanding real directions, so the
origin is a hyperbolic equilibrium of unstable dimension $n$, and
\Cref{thm:hyperbolic} gives directly
$h(\varphi_f,\{0\})\simeq S^n$, matching Theorem~A:
\[
h(\varphi_f,\{0\})\;\simeq\;S^n\;\simeq\;\Sigma S^{n-1}\;\simeq\;\Sigma F_f.
\]
\end{example}

\begin{example}[The $A_k$ singularities]
\label{ex:Ak}
For $f=z_1^{k+1}+z_2^2+\dots+z_n^2$, \Cref{thm:wedge} gives
$\mu(f)=k$ and $F_f\simeq\bigvee_k S^{n-1}$, so
$\Sigma F_f\simeq\bigvee_k S^n$.  The mechanism of
Theorem~A produces the same wedge dynamically, as follows.

Take the morsification $f_t=z_1^{k+1}+z_2^2+\dots+z_n^2-tz_1$
(\Cref{def:morsification}), with critical point equations
\[
\frac{\partial f_t}{\partial z_1}=(k+1)z_1^k-t=0,
\qquad
\frac{\partial f_t}{\partial z_j}=2z_j=0,\quad j=2,\dots,n.
\]
So $z_2=\dots=z_n=0$ and $z_1^k=t/(k+1)$; for $t\neq0$ small this
has exactly $k$ distinct roots $a_1(t),\dots,a_k(t)$, giving $k$
critical points $p_j(t)=(a_j(t),0,\dots,0)$.  Each is a nondegenerate
critical point of the holomorphic $f_t$: the diagonal entries of the
holomorphic Hessian are $k(k+1)a_j(t)^{k-1}\neq0$ (since
$a_j(t)\neq0$) and $2\neq0$.  By \Cref{lem:index-n} each $p_j(t)$ is a
nondegenerate critical point of $\re f_t$ of Morse index exactly
$n$; the Cauchy--Riemann structure forces this index.  The critical values are explicit: since
$a_j^{k}=t/(k+1)$,
\[
w_j(t)=f_t(p_j(t))=a_j^{k+1}-ta_j=-\frac{k}{k+1}\,t\,a_j(t),
\]
and with $a_j=\omega^j\bigl(t/(k+1)\bigr)^{1/k}$, $\omega=e^{2\pi i/k}$,
\[
w_i-w_j=-\frac{k}{k+1}(k+1)^{-1/k}\,(\omega^i-\omega^j)\,
t^{1+1/k}\ne0\qquad(i\ne j).
\]
So for any fixed small $t_*\neq0$ all critical values are distinct,
and only finitely many angles are excluded by \Cref{lem:rotation}.  For
$\theta_*$ outside this set and
$F=e^{-i\theta_*}f_{t_*}$, \Cref{lem:connections} excludes connecting
orbits and the sum property gives
\[
h\bigl(\varphi_F,S\bigr)\;\simeq\;\bigvee_{j=1}^k S^n,
\]
and continuation back to $f$ (Steps~1--4) yields
\[
h(\varphi_f,\{0\})\;\simeq\;\bigvee_k S^n\;\simeq\;\Sigma F_f.
\]
\end{example}

\begin{remark}[Simple singularities]
For the simple hypersurface singularities in their suspended normal
forms
\[
\begin{array}{ll}
A_k:\ z_1^{k+1}+z_2^2+\dots+z_n^2, & \mu=k,\\[2pt]
D_k:\ z_1^{2}z_2+z_2^{k-1}+z_3^2+\dots+z_n^2, & \mu=k,\\[2pt]
E_6:\ z_1^{3}+z_2^{4}+z_3^2+\dots+z_n^2, & \mu=6,\\[2pt]
E_7:\ z_1^{3}+z_1z_2^{3}+z_3^2+\dots+z_n^2, & \mu=7,\\[2pt]
E_8:\ z_1^{3}+z_2^{5}+z_3^2+\dots+z_n^2, & \mu=8,
\end{array}
\]
a morsification has exactly $\mu$ nondegenerate critical points, each
of real Morse index $n$ by \Cref{lem:index-n}, and Theorem~A gives
$h(\varphi_f,\{0\})\simeq\bigvee_\mu S^n\simeq\Sigma F_f$.
\end{remark}

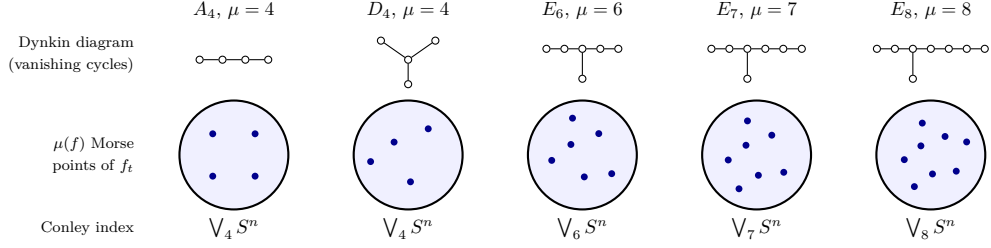
\begin{figure}[htb]
\centering
\begin{tikzpicture}[scale=0.72,transform shape,>=Stealth,nd/.style={circle,draw,fill=white,inner sep=1.3pt}]\begin{scope}[xshift=0.0cm]
\node at (0,2.45) {\small $A_4$, $\mu=4$};
\node[nd] (d0) at (-0.63,1.55) {};
\node[nd] (d1) at (-0.21,1.55) {};
\node[nd] (d2) at (0.21,1.55) {};
\node[nd] (d3) at (0.63,1.55) {};
\draw (d0)--(d1);
\draw (d1)--(d2);
\draw (d2)--(d3);
\fill[blue!6] (0,-0.2) circle (1.0); \draw[thick] (0,-0.2) circle (1.0);
\fill[blue!55!black] ($(0,-0.2)+(45:0.55)$) circle (1.8pt);
\fill[blue!55!black] ($(0,-0.2)+(135:0.55)$) circle (1.8pt);
\fill[blue!55!black] ($(0,-0.2)+(225:0.55)$) circle (1.8pt);
\fill[blue!55!black] ($(0,-0.2)+(315:0.55)$) circle (1.8pt);
\node at (0,-1.55) {\small $\bigvee_{4}S^{n}$};
\end{scope}
\begin{scope}[xshift=3.2cm]
\node at (0,2.45) {\small $D_4$, $\mu=4$};
\node[nd] (d0) at (0,1.55) {}; \node[nd] (d1) at (-0.5,1.9) {}; \node[nd] (d2) at (0.5,1.9) {}; \node[nd] (d3) at (0,1.1) {};
\draw (d0)--(d1) (d0)--(d2) (d0)--(d3);
\fill[blue!6] (0,-0.2) circle (1.0); \draw[thick] (0,-0.2) circle (1.0);
\foreach \i in {1,...,4}{\fill[blue!55!black] ($(0,-0.2)+({137.5*\i}:{0.7*sqrt(\i/4)})$) circle (1.8pt);}
\node at (0,-1.55) {\small $\bigvee_{4}S^{n}$};
\end{scope}
\begin{scope}[xshift=6.4cm]
\node at (0,2.45) {\small $E_6$, $\mu=6$};
\node[nd] (d0) at (-0.66,1.75) {};
\node[nd] (d1) at (-0.33,1.75) {};
\node[nd] (d2) at (0.00,1.75) {};
\node[nd] (d3) at (0.33,1.75) {};
\node[nd] (d4) at (0.66,1.75) {};
\draw (d0)--(d1);
\draw (d1)--(d2);
\draw (d2)--(d3);
\draw (d3)--(d4);
\node[nd] (db) at (0.00,1.2) {}; \draw (d2)--(db);
\fill[blue!6] (0,-0.2) circle (1.0); \draw[thick] (0,-0.2) circle (1.0);
\foreach \i in {1,...,6}{\fill[blue!55!black] ($(0,-0.2)+({137.5*\i}:{0.7*sqrt(\i/6)})$) circle (1.8pt);}
\node at (0,-1.55) {\small $\bigvee_{6}S^{n}$};
\end{scope}
\begin{scope}[xshift=9.600000000000001cm]
\node at (0,2.45) {\small $E_7$, $\mu=7$};
\node[nd] (d0) at (-0.83,1.75) {};
\node[nd] (d1) at (-0.50,1.75) {};
\node[nd] (d2) at (-0.17,1.75) {};
\node[nd] (d3) at (0.16,1.75) {};
\node[nd] (d4) at (0.49,1.75) {};
\node[nd] (d5) at (0.83,1.75) {};
\draw (d0)--(d1);
\draw (d1)--(d2);
\draw (d2)--(d3);
\draw (d3)--(d4);
\draw (d4)--(d5);
\node[nd] (db) at (-0.17,1.2) {}; \draw (d2)--(db);
\fill[blue!6] (0,-0.2) circle (1.0); \draw[thick] (0,-0.2) circle (1.0);
\foreach \i in {1,...,7}{\fill[blue!55!black] ($(0,-0.2)+({137.5*\i}:{0.7*sqrt(\i/7)})$) circle (1.8pt);}
\node at (0,-1.55) {\small $\bigvee_{7}S^{n}$};
\end{scope}
\begin{scope}[xshift=12.8cm]
\node at (0,2.45) {\small $E_8$, $\mu=8$};
\node[nd] (d0) at (-0.99,1.75) {};
\node[nd] (d1) at (-0.66,1.75) {};
\node[nd] (d2) at (-0.33,1.75) {};
\node[nd] (d3) at (0.00,1.75) {};
\node[nd] (d4) at (0.33,1.75) {};
\node[nd] (d5) at (0.66,1.75) {};
\node[nd] (d6) at (0.99,1.75) {};
\draw (d0)--(d1);
\draw (d1)--(d2);
\draw (d2)--(d3);
\draw (d3)--(d4);
\draw (d4)--(d5);
\draw (d5)--(d6);
\node[nd] (db) at (-0.33,1.2) {}; \draw (d2)--(db);
\fill[blue!6] (0,-0.2) circle (1.0); \draw[thick] (0,-0.2) circle (1.0);
\foreach \i in {1,...,8}{\fill[blue!55!black] ($(0,-0.2)+({137.5*\i}:{0.7*sqrt(\i/8)})$) circle (1.8pt);}
\node at (0,-1.55) {\small $\bigvee_{8}S^{n}$};
\end{scope}
\node[left,align=right] at (-1.7,1.65) {\scriptsize Dynkin diagram\\[-2pt]\scriptsize (vanishing cycles)};
\node[left,align=right] at (-1.7,-0.2) {\scriptsize $\mu(f)$ Morse\\[-2pt]\scriptsize points of $f_t$};
\node[left] at (-1.7,-1.55) {\scriptsize Conley index};
\end{tikzpicture}
\caption{Simple singularities: the index $\bigvee_\mu S^n$ depends only on $\mu$, not on the Dynkin diagram.}
\label{fig:simple}
\end{figure}

\section{Conclusion}
\label{sec:conclusion}

Theorem~A shows that the Conley index of the negative gradient flow
of $\re f$ at an isolated holomorphic singularity is the suspension of
the Milnor fiber, for every $n\ge1$ and any Riemannian metric.  The
proof proceeds by continuation to a morsification, in which the first
integral $\im f_t$ excludes connecting orbits after a rotation of the
target through an angle avoiding finitely many values.  Its corollaries
recover the Poincar\'e--Hopf index of $-\nabla\re f$ via McCord's
theorem \cite{McCord1989} and identify $\CH_n(\varphi_f,\{0\};\Z)$,
up to choices, with $\wtH_{n-1}(F_f;\Z)$, the group spanned by the
vanishing cycles of a morsification.

\bibliographystyle{alpha}

\end{document}